\documentclass[11pt,reqno]{article}

\usepackage[T1]{fontenc}
\usepackage{lmodern}
\usepackage{amsmath,amssymb,amsthm,amsfonts}
\usepackage{microtype}
\usepackage{geometry}
\usepackage{hyperref}
\hypersetup{colorlinks=true,linkcolor=blue,citecolor=blue,urlcolor=blue}

\newcommand{\seqnum}[1]{\href{https://oeis.org/#1}{#1}}

\theoremstyle{plain}
\newtheorem{theorem}{Theorem}
\newtheorem{corollary}[theorem]{Corollary}
\newtheorem{lemma}[theorem]{Lemma}
\newtheorem{proposition}[theorem]{Proposition}

\theoremstyle{definition}
\newtheorem{definition}[theorem]{Definition}
\newtheorem{example}[theorem]{Example}
\newtheorem{conjecture}[theorem]{Conjecture}

\theoremstyle{remark}
\newtheorem{remark}[theorem]{Remark}

\title{Counting Residue-Class Survivor Sets: An Asymptotic Comparison with Prime-Admissible Sets}
\author{Mario Raso\thanks{Department of Computer Science, Sapienza University of Rome, Viale Regina Elena 295, 00161 Rome, Italy. Email: \href{mailto:raso@di.uniroma1.it}{raso@di.uniroma1.it}.}
\and Daniele Venturi\thanks{Department of Computer Science, Sapienza University of Rome, Viale Regina Elena 295, 00161 Rome, Italy. Email: \href{mailto:venturi@di.uniroma1.it}{venturi@di.uniroma1.it}.}}
\date{}

\begin{document}
\maketitle

\begin{abstract}
For $n\geq1$, choose one residue class $r_k\pmod{k}$ for every
$2\leq k\leq n$, and retain those $m\in\{2,\ldots,n+1\}$ for which
$m\not\equiv r_k\pmod{k}$ whenever $k<m$.  Let $N(n)$ be the number
of distinct survivor sets obtained in this way.  We prove
\[
\log 2\leq
\liminf_{n\to\infty}\frac{\log N(n)}{n/\log n}
\leq
\limsup_{n\to\infty}\frac{\log N(n)}{n/\log n}
\leq2\log 2.
\]
If $A_{\rm adm}(n)$ denotes the number of subsets of $\{2,\ldots,n+1\}$ that omit at least one residue class modulo every prime, then our main comparison is
\[
\log N(n)=\log A_{\rm adm}(n)+o\!\left(\frac{n}{\log n}\right).
\]
We also obtain an asymptotic formula for the number of distinct intersections of survivor sets with the primes, and an exact recurrence $N(n+1)=N(n)+E(n)$, where $E(n)$ counts the survivor sets that can be extended by the new point $n+2$.
\end{abstract}



\section{Introduction}
\label{sec:introduction}

For an integer $n\geq1$, set $X_n=\{2,\ldots,n+1\}$.  Throughout, $\mathbb P$ denotes the set of primes. A truncated residue profile is a tuple
\[
r=(r_2,\ldots,r_n),\qquad r_k\in\mathbb Z/k\mathbb Z,
\]
and its survivor set is
\[
S_n(r)=\{m\in X_n:m\not\equiv r_k\pmod{k}
\text{ for every }2\leq k<m\}.
\]
Thus the class chosen modulo $k$ acts only on candidates larger than
$k$.  Define
\[
\mathcal F_n=\{S_n(r):r=(r_2,\ldots,r_n)\},
\qquad N(n)=|\mathcal F_n|.
\]
The counting problem concerns distinct outputs of the map
$ (r_2,\ldots,r_n)\mapsto S_n(r)$, rather than the number of profiles.

The construction lies near two established arithmetic settings. Covering systems and uncovered residue-class sets study unions and complements of congruence classes; the cited work includes quantitative bounds for uncovered sets, restrictions on distinct covering systems, and related finite-interval questions \cite{FilasetaFordKonyaginPomeranceYu,BalisterEtAlDensity,Hough,HoughNielsen,ZillerMorack}.
Those results provide context rather than estimates for $N(n)$: here every modulus $2,\ldots,n$ receives a class, the class modulo $k$ becomes active only above $k$, and all distinct finite outputs are counted.

A set $A\subseteq X_n$ is prime-admissible when it omits at least one residue class modulo every prime.  Write $A_{\rm adm}(n)$ for the number of such subsets. These sets form a downward-closed family. By contrast, $\mathcal F_n$ is not downward-closed: the restriction that modulus \(k\) acts only on candidates larger than \(k\), together with the composite moduli, imposes additional structure. Our main result is the comparison
\[
\log N(n)=\log A_{\rm adm}(n)+o\!\left(\frac{n}{\log n}\right).
\]
The proof uses an auxiliary downward-closed family determined by the prime-modulus restrictions. Composite moduli then allow prescribed large elements to be removed with subexponential loss on the \(n/\log n\) scale.

Several exact and asymptotic results accompany this comparison.
The number $P(n)$ of locally distinct profiles has a closed product formula, and inclusion--exclusion gives an exact formula for the number of profiles realizing any prescribed survivor set.  We prove
\[
\log 2
\leq\liminf_{n\to\infty}\frac{\log N(n)}{n/\log n}
\leq\limsup_{n\to\infty}\frac{\log N(n)}{n/\log n}
\leq2\log 2.
\]
If $N_{\mathbb P}(n)$ counts the distinct sets $S_n(r)\cap\mathbb P$,
then
\[
\log N_{\mathbb P}(n)\sim(\log2)\frac{n}{\log n}.
\]
We also characterize extendible survivor sets and prove the exact recurrence
\[
N(n+1)=N(n)+E(n).
\]
The recurrence also provides an exact enumeration scheme for the initial
values.

For comparison with symbolic dynamics, the support of a length-$n$ block of the full $\mathbb P$-admissible subshift is exactly a prime-admissible subset of an interval of length $n$. Kasjan, Lema\'nczyk, and Zuniga Alterman proved unconditional block-complexity bounds with respective bases
2 and 4 \cite[Theorem~1.1]{KasjanLemanczykZuniga}. Translation and zero-extension identify their finite block family with the family counted by $A_{\rm adm}(n)$; it is the full admissible subshift, rather than the orbit closure of the prime indicator, that is used here. The asymptotic comparison above therefore connects the survivor-set model to that block-complexity problem.

The sequence \(N(n)\) arose from the first author's doctoral research and was subsequently submitted by him to the OEIS, where it is recorded as entry \seqnum{A396595}; see \cite{OEISA396595}. The underlying construction is also discussed in his doctoral thesis \cite{RasoThesis}. The present paper develops the structural and asymptotic study of this sequence.
Appendix \ref{subsec:global-profile} gives a global formulation in which the finite survivor sets arise by restriction.

Sections~\ref{sec:local-profiles}--\ref{sec:asymptotic-bounds} develop the local count, the exact representation formula, and the asymptotic bounds.
Sections~\ref{sec:prime-traces}--\ref{sec:algorithms} treat prime intersections, the comparison with prime-admissible sets, extendibility, and exact computation.  The appendices discuss global profiles and computational reproducibility.

\section{Local profiles and the explicit count \texorpdfstring{\(P(n)\)}{P(n)}}
\label{sec:local-profiles}
We begin with the local behavior of each modulus.

\begin{definition}
For fixed integers $n\ge 1$ and $2\le k\le n$, and for a residue class $a\in \mathbb Z/k\mathbb Z$, define
\[
E^{(n)}_{k,a}:=\{m\in\{k+1,\dots,n+1\}: m\equiv a \pmod{k}\}.
\]
This is the set of elements of $\{2,\dots,n+1\}$ eliminated by the single congruence
\[
m\equiv a\pmod{k}.
\]
\end{definition}

\begin{lemma}\label{lem:local-count}
For fixed integers $n\ge 1$ and $2\le k\le n$, the number of distinct sets among
\[
E^{(n)}_{k,a}, \qquad a\in \mathbb Z/k\mathbb Z,
\]
is
\[
c_{k,n}=\min(k,n-k+2).
\]
\end{lemma}

\begin{proof}
Suppose first that
\[
k\le \frac{n+1}{2}.
\]
Then the interval $\{k+1,\dots,n+1\}$ has cardinality
\(n-k+1\ge k,\)
so it contains at least one complete system of residues modulo $k$. Hence every residue class modulo $k$ occurs at least once in this interval, and the sets $E^{(n)}_{k,a}$ are pairwise distinct. Therefore
\[
c_{k,n}=k
\]
in this case.

Now suppose that
\[
k> \frac{n+1}{2}.
\]
Then the interval $\{k+1,\dots,n+1\}$ has cardinality
\[
n-k+1<k.
\]
Hence two distinct elements of this interval cannot be congruent modulo $k$. Therefore each residue class that occurs gives a singleton set, while every residue class that does not occur gives the empty set. Since exactly $n-k+1$ residue classes occur, the number of distinct sets is
\[
(n-k+1)+1=n-k+2.
\]
Thus in this case
\[
c_{k,n}=n-k+2.
\]
Combining the two cases yields
\[
c_{k,n}=\min(k,n-k+2),
\]
as claimed.
\end{proof}

\begin{definition}
Two truncated profiles
\[
r=(r_2,\dots,r_n), \qquad s=(s_2,\dots,s_n),
\]
are said to be \emph{locally equivalent} if
\[
E^{(n)}_{k,r_k}=E^{(n)}_{k,s_k}
\qquad\text{for every }2\le k\le n.
\]
\end{definition}

\begin{theorem}\label{thm:Pn}
Let $P(n)$ denote the number of local equivalence classes of truncated profiles of length $n$. Then
\[
P(n)=\prod_{k=2}^{n}\min(k,n-k+2).
\]
Equivalently,
\[
P(n)=
\begin{cases}
(m!)^2, & \text{if } n=2m-1,\\[4pt]
m!(m+1)!, & \text{if } n=2m.
\end{cases}
\]
\end{theorem}

\begin{proof}
By Lemma~\ref{lem:local-count}, for each fixed $k$ there are exactly
\(c_{k,n}=\min(k,n-k+2)\)
distinct local choices. Since these choices are independent for different $k$, the total number of local equivalence classes is
\[
P(n)=\prod_{k=2}^{n} c_{k,n}
=\prod_{k=2}^{n}\min(k,n-k+2).
\]

It remains to simplify the product. If $n=2m-1$, then
\[
\min(k,2m-k+1)=
\begin{cases}
k, & 2\le k\le m,\\
2m-k+1, & m+1\le k\le 2m-1.
\end{cases}
\]
Hence
\[
P(2m-1)
=\Bigl(\prod_{k=2}^{m} k\Bigr)\Bigl(\prod_{k=m+1}^{2m-1}(2m-k+1)\Bigr)
=(2\cdot 3\cdots m)(m\cdot (m-1)\cdots 2)
=(m!)^2.
\]

If $n=2m$, then
\[
\min(k,2m-k+2)=
\begin{cases}
k, & 2\le k\le m+1,\\
2m-k+2, & m+2\le k\le 2m.
\end{cases}
\]
Thus
\[
P(2m)
=\Bigl(\prod_{k=2}^{m+1} k\Bigr)\Bigl(\prod_{k=m+2}^{2m}(2m-k+2)\Bigr)
=(2\cdot 3\cdots (m+1))(m\cdot (m-1)\cdots 2)
=m!(m+1)!.
\]
This completes the proof.
\end{proof}

\begin{remark}
The explicit sequence \(P(n)\) satisfies
\(P(n)=\seqnum{A010551}(n+1)\); see \cite{OEISA010551}.  Here it is used as a local upper bound for the distinct-output count $N(n)$. The latter is closer in form to counting questions for covering systems and uncovered residue-class structures; compare
\cite{BalisterEtAlStructure,FilasetaFordKonyaginPomeranceYu}.
\end{remark}

\section{An exact representation formula for \texorpdfstring{\(N(n)\)}{N(n)}}
\label{sec:exact-representation}

Recall that \(\mathcal F_n\) is the family of survivor sets and that
\(N(n)=|\mathcal F_n|\).

\begin{definition}
Let
\[
X_n:=\{2,3,\dots,n+1\}.
\]
For \(B\subseteq X_n\) and \(2\le k\le n\), define
\[
\nu_k(B):=
\#\{b\bmod k:\ b\in B,\ b>k\}.
\]
For \(A\subseteq X_n\), put
\[
R_n(A):=
\sum_{Y\subseteq X_n\setminus A}
(-1)^{|Y|}
\prod_{k=2}^{n}
\bigl(k-\nu_k(A\cup Y)\bigr).
\]
\end{definition}

\begin{theorem}[An exact formula for \(N(n)\)]\label{thm:exact-Nn}
For every integer \(n\ge 1\) and every subset \(A\subseteq X_n\), the integer
\(R_n(A)\) is exactly the number of truncated profiles
\[
r=(r_2,\dots,r_n),\qquad r_k\in \mathbb Z/k\mathbb Z,
\]
such that
\[
S_n(r)=A.
\]
Consequently,
\[
A\in\mathcal F_n
\quad\Longleftrightarrow\quad
R_n(A)>0,
\]
and hence
\[
N(n)=
\sum_{A\subseteq X_n}
\min\{1,R_n(A)\}.
\]
\end{theorem}

\begin{proof}
Fix \(A\subseteq X_n\). Let \(B=X_n\setminus A\). We first count profiles for which all elements of \(A\) survive. For a fixed modulus \(k\), the residue \(r_k\) is forbidden to lie in any residue class represented modulo \(k\) by
an element \(a\in A\) with \(a>k\). The number of forbidden residue classes is
\(\nu_k(A),\)
so the number of admissible choices for \(r_k\) is
\(k-\nu_k(A).\)

More generally, if \(Y\subseteq B\), the number of profiles for which every element of \(A\cup Y\) survives is
\[
\prod_{k=2}^{n}\bigl(k-\nu_k(A\cup Y)\bigr).
\]
Indeed, for each \(k\), one must avoid precisely the residue classes modulo \(k\) represented by elements of \(A\cup Y\) that are larger than \(k\).

Now use inclusion--exclusion over the elements of \(B\). Starting from the profiles for which all elements of \(A\) survive, we exclude those for which at least one element of \(B\) also survives. Thus the number of profiles for which exactly the elements of \(A\) survive is
\[
\sum_{Y\subseteq B}
(-1)^{|Y|}
\prod_{k=2}^{n}
\bigl(k-\nu_k(A\cup Y)\bigr)
=
R_n(A).
\]
Therefore \(R_n(A)\) counts exactly the profiles \(r\) with \(S_n(r)=A\).

It follows immediately that \(A\in\mathcal F_n\) if and only if at least one profile realizes \(A\), which is equivalent to \(R_n(A)>0\). Since \(R_n(A)\) is a nonnegative integer, the number of distinct survivor sets is
\[
N(n)=
\sum_{A\subseteq X_n}
\min\{1,R_n(A)\}.
\]
This proves the formula.
\end{proof}

\begin{proposition}\label{prop:NleP}
For every integer $n\ge 1$,
\[
N(n)\le P(n).
\]
\end{proposition}

\begin{proof}
Each truncated profile determines a local equivalence class, and each local equivalence class determines a unique family of forbidden sets
\[
E^{(n)}_{k,r_k} \qquad (2\le k\le n).
\]
Hence it determines a unique survivor set $S_n(r)$. Therefore the number of distinct survivor sets cannot exceed the number of local equivalence classes.
\end{proof}

\section{Asymptotic bounds for \texorpdfstring{\(N(n)\)}{N(n)}}
\label{sec:asymptotic-bounds}
\begin{theorem}[Lower bound for \(N(n)\)]\label{thm:Nn-lower-bound}
As \(n\to\infty\),
\[
N(n)\ge
\exp\!\left((\log 2+o(1))\frac{n}{\log n}\right).
\]
Equivalently,
\[
\log N(n)\ge
(\log 2+o(1))\frac{n}{\log n}.
\]
\end{theorem}

\begin{proof}
Fix an integer \(J\ge 2\).  Let \(\mathcal G_{n,J}\) be the set of primes
\(p\le n+1\) such that
\[
p>\frac{n}{J}+1
\]
and such that, for every integer \(t\) with \(2\le t\le J-1\), the integer
\(t(p-1)+1\)
is composite.

By the prime number theorem,
\[
\#\left\{p\le n+1:\ p>\frac{n}{J}+1\right\}
=
\left(1-\frac{1}{J}+o(1)\right)\frac{n}{\log n}.
\]
We must remove those primes \(p\) for which \(t(p-1)+1\) is also prime for at least one \(2\le t\le J-1\).  For each fixed \(t\), apply the upper-bound sieve of Koukoulopoulos~\cite[Theorem~18.11(b), p.~190]{Koukoulopoulos} to the two affine linear forms
\[
L_1(u)=u,
\qquad
L_2(u)=tu-(t-1);
\]
see \cite[Example~18.3, p.~183, and Theorem~18.11(b), p.~190]{Koukoulopoulos}.
For every prime \(\ell\nmid t(t-1)\), the product \(L_1(u)L_2(u)\) has two distinct roots modulo \(\ell\); the finitely many exceptional primes contribute only local factors depending on \(t\). In the notation of that theorem, take \(D=(n+1)/(\log(n+1))^5\) and
\(z=(n+1)^{1/10}\).  The polynomial remainder is
\(O(\nu(d))\), and \(\sum_{d\le D}\nu(d)\ll D\log D\);
the sifted pairs with a prime factor at most \(z\) contribute \(O(z)\).
The theorem therefore gives, for fixed \(t\), that the number of such primes is
\[
O_t\!\left(\frac{n}{(\log n)^2}\right).
\]
Since \(J\) is fixed, the total number of excluded primes is
\[
O_J\!\left(\frac{n}{(\log n)^2}\right)
=
o\!\left(\frac{n}{\log n}\right).
\]
Therefore
\[
|\mathcal G_{n,J}|
=
\left(1-\frac{1}{J}+o(1)\right)\frac{n}{\log n}.
\]

For each subset \(T\subseteq \mathcal G_{n,J}\), define a truncated profile
\(r^{(T)}=(r^{(T)}_2,\dots,r^{(T)}_n)\) by
\[
r^{(T)}_k\equiv
\begin{cases}
1 \pmod{k}, & \text{if } k=p-1 \text{ for some } p\in T,\\
0 \pmod{k}, & \text{otherwise.}
\end{cases}
\]
We claim that different subsets \(T\subseteq \mathcal G_{n,J}\) give different survivor sets.

First let \(m\in X_n\) be composite. Choose a prime divisor \(q\) of \(m\) with \(q<m\).  Since \(q\) is prime, the equality \(q=p-1\) with \(p\in\mathcal G_{n,J}\) is impossible for all sufficiently large \(n\): if \(q=2\), then \(p=3\), while if \(q\ge 3\), then \(q+1\) is an even integer greater than \(2\).  Thus \(r^{(T)}_q\equiv 0\pmod q\).  Since \(m\equiv 0\pmod q\), the integer \(m\) is eliminated.

Now let \(p\le n+1\) be prime.  If \(p\in T\), then \(p\) is eliminated by the modulus \(p-1\), because \(p\equiv 1\pmod{p-1}\).  If \(p\notin T\), then \(p\) is not eliminated by any modulus \(k<p\) with \(r^{(T)}_k\equiv 0\pmod k\), since this would require \(k\mid p\).  It also cannot be eliminated by a modulus \(q-1\) with \(q\in T\).  Indeed, such an elimination would imply \(p\equiv 1\pmod{q-1},\)
so
\(p=t(q-1)+1\)
for some integer \(t\ge 1\).  Since \(q>\frac{n}{J}+1\), we have \(q-1>n/J\).  The inequality \(p\le n+1\) therefore gives \(t<J\), hence \(t\le J-1\).  If \(t=1\), then \(p=q\), contradicting \(p\notin T\). If
\(2\le t\le J-1\), then the definition of \(\mathcal G_{n,J}\) says that \(t(q-1)+1\) is composite, again impossible. Thus \(p\) is not eliminated.

Consequently the primes in \(\mathcal G_{n,J}\) are independently switched off by membership in \(T\). Distinct subsets \(T\subseteq\mathcal G_{n,J}\) give distinct survivor sets, and hence
\[
N(n)\ge 2^{|\mathcal G_{n,J}|}.
\]
Therefore, for every fixed \(J\ge 2\),
\[
\liminf_{n\to\infty}
\frac{\log N(n)}{n/\log n}
\ge
\left(1-\frac{1}{J}\right)\log 2.
\]
Since \(J\) is arbitrary, letting \(J\to\infty\) yields
\[
\liminf_{n\to\infty}
\frac{\log N(n)}{n/\log n}
\ge \log 2,
\]
which is equivalent to the stated lower bound.
\end{proof}

\begin{lemma}[One-class large sieve bound]\label{lem:one-class-large-sieve}
Let \(Q\ge 2\), and let \(I\) be an interval containing \(x\) consecutive integers. For each prime \(p\le Q\), let \(a_p\) be a residue class modulo \(p\).  Then, uniformly in \(I\) and in the choices of the classes \(a_p\),
\[
\#\{m\in I:\ m\not\equiv a_p\pmod p
\text{ for every prime }p\le Q\}
\le
(1+o(1))\frac{x+Q^2}{\log Q},
\]
as \(Q\to\infty\).
\end{lemma}

\begin{proof}
Apply Kowalski's abstract sieve inequality \cite[Proposition~2.3, p.~12]{KowalskiLargeSieve} to the classical sieve setting
\(\mathbb Z\to\mathbb Z/p\mathbb Z\), with sieving set
\(\Omega_p=\{a_p\}\) for each prime \(p\le Q\), and take as sieve support
the squarefree integers \(q\le Q\).
Since
\[
\frac{|\Omega_p|/p}{1-|\Omega_p|/p}
=
\frac{1}{p-1},
\]
the sieve denominator is
\[
G(Q)
:=
\sum_{q\le Q}\mu^2(q)
\prod_{p\mid q}\frac{1}{p-1}
=
\sum_{q\le Q}\frac{\mu^2(q)}{\varphi(q)}.
\]
For an interval of \(x\) consecutive integers, the classical one-dimensional large-sieve inequality gives
\(\Delta\le x-1+Q^2;\)
see Kowalski~\cite[Theorem~4.1, pp.~48--49]{KowalskiLargeSieve}.
Therefore
\[
\#\{m\in I:\ m\not\equiv a_p\pmod p
\text{ for every prime }p\le Q\}
\le
\frac{x-1+Q^2}{G(Q)}.
\]
By Koukoulopoulos~\cite[Exercise~14.3(b), p.~154]{Koukoulopoulos},
\[
G(Q)
=
\sum_{q\le Q}\frac{\mu^2(q)}{\varphi(q)}
=
\log Q+O(1).
\]
Hence
\[
\frac{x-1+Q^2}{G(Q)}
\le
(1+o(1))\frac{x+Q^2}{\log Q},
\]
as \(Q\to\infty\).  The estimate is uniform in the position of \(I\)
and in the choices of the residue classes \(a_p\).
\end{proof}

\begin{theorem}[Upper bound for \(N(n)\)]\label{thm:Nn-upper-bound}
There is an absolute constant \(C>0\) such that, for all sufficiently large
\(n\),
\[
N(n)\le
\exp\!\left(C\frac{n}{\log n}\right).
\]
Equivalently,
\[
\log N(n)=O\!\left(\frac{n}{\log n}\right).
\]
\end{theorem}

\begin{proof}
Put
\[
Q=\left\lfloor \frac{n^{1/2}}{\log n}\right\rfloor .
\]
First fix the prime-modulus residues
\[
r_p\pmod p
\qquad (p\le Q,\ p\text{ prime}).
\]
The number of possible choices of these residues is
\[
\prod_{p\le Q}p=\exp(\vartheta(Q)),
\]
where
\[
\vartheta(Q):=\sum_{p\le Q}\log p.
\]
By Chebyshev's estimate, \(\vartheta(Q)=O(Q)\). Hence
\[
\vartheta(Q)=o\!\left(\frac{n}{\log n}\right).
\]

Now fix one choice of the residues \(r_p\pmod p\) for primes \(p\le Q\). Any survivor \(m>Q+1\) must satisfy
\[
m\not\equiv r_p\pmod p
\qquad (p\le Q,\ p\text{ prime}),
\]
because every such prime modulus \(p\) is smaller than \(m\). Therefore the large part of any survivor set is contained in
\[
A(r;Q):=
\{m\in\{Q+2,\dots,n+1\}:\ m\not\equiv r_p\pmod p
\text{ for every prime }p\le Q\}.
\]
By Lemma~\ref{lem:one-class-large-sieve},
\[
|A(r;Q)|
\ll
\frac{n+Q^2}{\log Q}.
\]
Since \(Q^2=o(n)\) and
\[
\log Q=\left(\frac{1}{2}+o(1)\right)\log n,
\]
we have
\[
|A(r;Q)|\ll \frac{n}{\log n}.
\]
Thus, for this fixed choice of prime residues, the part of a survivor set lying
in \(\{Q+2,\dots,n+1\}\) has at most
\(2^{O(n/\log n)}\)
possibilities.

The initial part \(\{2,\dots,Q+1\}\) has size \(Q\), and therefore contributes
at most
\[
2^Q=\exp(O(Q))
=
\exp\!\left(o\!\left(\frac{n}{\log n}\right)\right)
\]
possibilities.  Hence, for each fixed choice of the prime residues, the number
of possible survivor sets is at most
\[
\exp\!\left(O\!\left(\frac{n}{\log n}\right)\right).
\]
Multiplying by the number \(\exp(\vartheta(Q))\) of choices of the prime
residues gives
\[
N(n)
\le
\exp(\vartheta(Q))
\exp\!\left(O\!\left(\frac{n}{\log n}\right)\right)
=
\exp\!\left(O\!\left(\frac{n}{\log n}\right)\right),
\]
because \(\vartheta(Q)=o(n/\log n)\).  This proves the theorem.
\end{proof}

\begin{corollary}[Correct order of growth]\label{cor:Nn-correct-order}
As \(n\to\infty\),
\[
\log N(n)\asymp \frac{n}{\log n}.
\]
More precisely,
\[
\exp\!\left((\log 2+o(1))\frac{n}{\log n}\right)
\le
N(n)
\le
\exp\!\left(O\!\left(\frac{n}{\log n}\right)\right).
\]
\end{corollary}

\begin{proof}
This follows immediately from Theorems~\ref{thm:Nn-lower-bound}
and~\ref{thm:Nn-upper-bound}.
\end{proof}

\begin{corollary}[A constant window]\label{cor:constant-window}
One has
\[
\log 2
\le
\liminf_{n\to\infty}\frac{\log N(n)}{n/\log n}
\le
\limsup_{n\to\infty}\frac{\log N(n)}{n/\log n}
\le
2\log 2.
\]
In particular, if the limit
\[
C_N:=\lim_{n\to\infty}\frac{\log N(n)}{n/\log n}
\]
exists, then
\[
C_N\in[\log 2,2\log 2].
\]
\end{corollary}

\begin{proof}
The lower bound follows from Theorem~\ref{thm:Nn-lower-bound}.  For the upper
bound, choose
\[
Q=\left\lfloor \frac{n^{1/2}}{L(n)}\right\rfloor,
\]
where \(L(n)\to\infty\) and \(\log L(n)=o(\log n)\).  Then
\(Q^2=o(n)\) and
\[
\log Q=\left(\frac{1}{2}+o(1)\right)\log n.
\]
In the proof of Theorem~\ref{thm:Nn-upper-bound}, Lemma~\ref{lem:one-class-large-sieve}
therefore gives, for each fixed choice of prime residues modulo primes
\(p\le Q\),
\[
|A(r;Q)|\le (2+o(1))\frac{n}{\log n}.
\]
Hence the large part of a survivor set has at most
\(2^{(2+o(1))n/\log n}\)
possibilities.  The choices of the prime residues up to \(Q\), and the initial
segment \(\{2,\ldots,Q+1\}\), contribute only
\[
\exp\!\left(o\!\left(\frac{n}{\log n}\right)\right)
\]
possibilities.  Thus
\[
N(n)\le
\exp\!\left((2\log 2+o(1))\frac{n}{\log n}\right).
\]
Combining this with the lower bound gives the stated inequalities.
\end{proof}

\section{Intersections with the primes}
\label{sec:prime-traces}
Let
\[
\Pi_n:=\{p\in X_n:p\text{ is prime}\}.
\]
Define
\[
N_{\mathbb P}(n):=
\#\{S_n(r)\cap \Pi_n:r=(r_2,\ldots,r_n)\}.
\]
Thus \(N_{\mathbb P}(n)\) counts the distinct intersections with the primes of all
survivor sets in \(\mathcal F_n\).

\begin{theorem}[Asymptotic growth of intersections with the primes]\label{thm:prime-trace-growth}
As \(n\to\infty\),
\[
\log N_{\mathbb P}(n)\sim
(\log 2)\frac{n}{\log n}.
\]
\end{theorem}

\begin{proof}
Since every such intersection is a subset of \(\Pi_n\), we have
\(N_{\mathbb P}(n)\le 2^{|\Pi_n|}.\)
By the prime number theorem,
\[
|\Pi_n|=\pi(n+1)=\left(1+o(1)\right)\frac{n}{\log n}.
\]
Therefore
\[
N_{\mathbb P}(n)
\le
\exp\!\left((\log 2+o(1))\frac{n}{\log n}\right).
\]

For the lower bound, we use the construction from
Theorem~\ref{thm:Nn-lower-bound}.  For each fixed \(J\ge2\), that construction
produces a set \(\mathcal G_{n,J}\) of primes with
\[
|\mathcal G_{n,J}|
=
\left(1-\frac{1}{J}+o(1)\right)\frac{n}{\log n},
\]
such that each subset \(T\subseteq\mathcal G_{n,J}\) gives a survivor set whose
intersection with \(\mathcal G_{n,J}\) is \(\mathcal G_{n,J}\setminus T\).
Thus different choices of \(T\) give different intersections with the primes, and hence
\[
N_{\mathbb P}(n)\ge 2^{|\mathcal G_{n,J}|}.
\]
It follows that, for every fixed \(J\ge2\),
\[
\liminf_{n\to\infty}
\frac{\log N_{\mathbb P}(n)}{n/\log n}
\ge
\left(1-\frac{1}{J}\right)\log 2.
\]
Letting \(J\to\infty\) gives the matching lower bound
\[
\liminf_{n\to\infty}
\frac{\log N_{\mathbb P}(n)}{n/\log n}
\ge\log2.
\]
Together with the upper bound, this proves the asymptotic formula.
\end{proof}

\section{Prime-admissible sets and asymptotic comparison}
\label{sec:admissible-equivalence}

We now compare \(\mathcal F_n\) with the downward-closed family of prime-admissible sets. A set
\(A\subseteq X_n\) is called \emph{prime-admissible} if
\[
A\bmod p\ne\mathbb Z/p\mathbb Z
\qquad\text{for every prime }p.
\]
Let
\[
\mathcal A_n:=
\{A\subseteq X_n:A\text{ is prime-admissible}\},
\qquad
A_{\rm adm}(n):=|\mathcal A_n|.
\]

To compare \(\mathcal F_n\) with the prime-admissible family
\(\mathcal A_n\), introduce the downward-closed family
\(\mathcal W_n\) generated by the prime-modulus restrictions.

For residue classes
\[
\mathbf a=(a_p)_{p\le n},
\qquad a_p\in\mathbb Z/p\mathbb Z,
\]
indexed by the primes \(p\le n\), define the set determined by the
prime-modulus restrictions
\[
P_n(\mathbf a):=
\{m\in X_n:
m\not\equiv a_p\pmod p
\text{ for every prime }p<m\}.
\]
Set
\[
\mathcal W_n:=
\{A\subseteq X_n:
A\subseteq P_n(\mathbf a)
\text{ for some }\mathbf a\}.
\]

\begin{lemma}\label{lem:prime-containers}
For every \(n\ge1\),
\[
\mathcal F_n\subseteq\mathcal W_n.
\]
Moreover, every set \(P_n(\mathbf a)\) belongs to \(\mathcal F_n\).
\end{lemma}

\begin{proof}
If \(S_n(r)\in\mathcal F_n\), then retaining only the coordinates \(r_p\)
with \(p\) prime gives
\[
S_n(r)\subseteq P_n((r_p)_{p\le n}),
\]
and hence \(S_n(r)\in\mathcal W_n\).

Conversely, fix \(\mathbf a\).  For a prime modulus \(p\), put
\(r_p=a_p\).  For each composite \(k\le n\), choose a prime divisor
\(p(k)\mid k\) and choose \(r_k\pmod k\) satisfying
\[
r_k\equiv a_{p(k)}\pmod{p(k)}.
\]
If \(m\equiv r_k\pmod k\), then
\(m\equiv a_{p(k)}\pmod{p(k)}\); thus the composite coordinate eliminates
only integers already eliminated by the prime coordinate \(p(k)\).
It follows that
\[
S_n(r)=P_n(\mathbf a).
\]
\end{proof}

For later use, put
\[
K(n):=\max\{|S|:S\in\mathcal F_n\}.
\]
The argument in the proof of Theorem~\ref{thm:Nn-upper-bound}, applied to
one survivor set rather than to the whole family, gives
\[
K(n)\ll\frac{n}{\log n}.
\]

\begin{lemma}\label{lem:admissible-envelope}
As \(n\to\infty\),
\[
A_{\rm adm}(n)
\le
|\mathcal W_n|
\le
A_{\rm adm}(n)
\exp\!\left(
o\!\left(\frac{n}{\log n}\right)
\right).
\]
Equivalently,
\[
\log|\mathcal W_n|
=
\log A_{\rm adm}(n)
+
o\!\left(\frac{n}{\log n}\right).
\]
\end{lemma}

\begin{proof}
Every prime-admissible set belongs to \(\mathcal W_n\): for each prime \(p\),
choose a residue class \(a_p\pmod p\) omitted by the set. Hence
\(A_{\rm adm}(n)\le|\mathcal W_n|.\)

For the reverse comparison, let \(A\in\mathcal W_n\), and choose
\(\mathbf a\) such that \(A\subseteq P_n(\mathbf a)\). By
Lemma~\ref{lem:prime-containers}, \(P_n(\mathbf a)\in\mathcal F_n\), and
hence
\(|A|\le |P_n(\mathbf a)|\le K(n).\)
For every prime \(p\le |A|\), the class \(a_p\pmod p\) contains no element
of \(A\) greater than \(p\).  Among the elements of
\(A\cap\{2,\ldots,p\}\), that class contains at most one element.  Delete
this possible element for each prime \(p\le |A|\), and call the remaining
set \(B\).  Then
\[
|A\setminus B|\le\pi(|A|)\le\pi(K(n)).
\]
The set \(B\) omits a residue class modulo every prime \(p\le |A|\); for
\(p>|A|\), it cannot occupy all \(p\) residue classes because
\(|B|\le|A|<p\).  Thus \(B\in\mathcal A_n\).

Consequently,
\[
|\mathcal W_n|
\le
A_{\rm adm}(n)
\sum_{j\le\pi(K(n))}\binom{n}{j}.
\]
Since \(K(n)\ll n/\log n\),
\[
\pi(K(n))=O\!\left(\frac{n}{(\log n)^2}\right),
\]
and the standard bound for partial binomial sums gives
\[
\log\sum_{j\le\pi(K(n))}\binom{n}{j}
=
O\!\left(\frac{n\log\log n}{(\log n)^2}\right)
=
o\!\left(\frac{n}{\log n}\right).
\]
This proves the lemma.
\end{proof}

The following lemma shows that composite coordinates can delete prescribed
large elements of a prime container with only a subexponential loss on the
initial interval.

\begin{lemma}[Deleting prescribed elements with composite moduli]
\label{lem:private-composite-coordinates}
Fix \(0<\eta<1/2\), and put
\[
J=(\log n)^{1/2-\eta},
\qquad
Y=\left\lfloor\frac{n}{J}\right\rfloor.
\]
For all sufficiently large \(n\), let \(P\in\mathcal F_n\), and let
\[
B\subseteq P\cap\{Y+2,\ldots,n+1\}.
\]
Then there is an injective assignment
\[
b\longmapsto k_b
\qquad (b\in B)
\]
such that, for every \(b\in B\),
\[
\left\lceil\frac b2\right\rceil\le k_b<b,
\qquad
k_b\ \text{is composite},
\]
and
\[
c\not\equiv b\pmod{k_b}
\qquad
\text{for every }c\in P\setminus\{b\}.
\]
Consequently, suppose that \(r^{(0)}\) is a profile with
\(S_n(r^{(0)})=P\) whose composite coordinates satisfy the redundancy
property obtained from the construction in the proof of
Lemma~\ref{lem:prime-containers}: for every composite \(k\le n\), there
is a prime divisor \(p(k)\mid k\) such that
\[
r^{(0)}_k\equiv r^{(0)}_{p(k)}\pmod{p(k)}.
\]
If, for each \(b\in B\), the coordinate \(r^{(0)}_{k_b}\) is replaced by
\(b\pmod{k_b}\), while all other coordinates are left unchanged, then
the resulting profile \(r\) satisfies
\[
S_n(r)=P\setminus B.
\]
\end{lemma}

\begin{proof}
For \(b\in B\), let
\[
I_b:=
\left[\left\lceil\frac b2\right\rceil,b\right)\cap\mathbb Z.
\]
Every \(k\in I_b\) satisfies \(k<b\le n+1\), and hence \(k\le n\);
thus every candidate is a valid profile coordinate.  Call
\(k\in I_b\) a collision candidate if
\(c\equiv b\pmod k\)
for some \(c\in P\setminus\{b\}\).  We first bound the number of such
candidates uniformly in \(b\).

Suppose first that \(c<b\).  Then \(k\mid b-c\), and
\[
0<b-c<b,
\qquad
k\ge\frac b2.
\]
Writing \(b-c=qk\), we obtain \(1\le q<2\), and hence \(q=1\).
Thus \(k=b-c\), so each \(c<b\) contributes at most one collision
candidate.

Suppose next that \(c>b\).  Then
\(c-b=tk\)
for some positive integer \(t\).  Since \(c\le n+1\) and \(k\ge b/2\),
\[
t\le\frac{n+1-b}{k}
<
\frac{2(n+1)}{b}.
\]
Because \(b\ge Y+2\) and \(J>1\) for all sufficiently large \(n\), we
have \(b>n/J+1\), and therefore
\[
\frac{n+1}{b}<J.
\]
It follows that \(t<2J\).  For fixed \(c\) and \(t\), the equation
\(c-b=tk\) determines at most one integer \(k\).  Hence
\[
\#\{k\in I_b:k\text{ is a collision candidate}\}
\ll J|P|.
\]

Since \(P\in\mathcal F_n\), the definition of \(K(n)\) and the bound
\(K(n)\ll n/\log n\) give
\[
|P|\le K(n)\ll\frac{n}{\log n}.
\]
Moreover, \(b>n/J\), and consequently
\[
\frac{J|P|}{b}
\ll
\frac{J^2}{\log n}
=
\frac{1}{(\log n)^{2\eta}}
=o(1).
\]
Thus the collision candidates form an \(o(b)\)-subset of \(I_b\),
uniformly for \(b\in B\).

We must also exclude prime coordinates.  Since
\[
b>\frac{n}{J}
\]
and \(\log J=o(\log n)\), we have \(\log b\sim\log n\), uniformly for
\(b\in B\).  The prime number theorem therefore gives
\[
\#\{k\in I_b:k\text{ is prime}\}
\le\pi(b)
\ll\frac{b}{\log b}
=o(b).
\]
Finally, during a greedy construction, the number of coordinates
already assigned is at most \(|B|\le|P|\), and
\[
\frac{|P|}{b}
\ll
\frac{J}{\log n}
=
\frac{1}{(\log n)^{1/2+\eta}}
=o(1).
\]
On the other hand,
\[
|I_b|
=
b-\left\lceil\frac b2\right\rceil
=
\left\lfloor\frac b2\right\rfloor
=
\frac b2+O(1).
\]
Hence, after removing the collision candidates, the prime candidates,
and all previously assigned coordinates, the set \(I_b\) still contains
\[
\left(\frac12+o(1)\right)b
\]
available integers, uniformly for \(b\in B\).  Since
\(b>n/J\to\infty\), every nonprime candidate in \(I_b\) is composite
for all sufficiently large \(n\).  We may therefore choose the
coordinates \(k_b\) greedily so that they are distinct, composite, and
satisfy the required noncollision condition.

For the modified profile, each \(b\in B\) is eliminated by its assigned
coordinate \(k_b\).  If \(c\in P\setminus B\), a modified coordinate
\(k_b\) is inactive when \(c\le k_b\), while for \(c>k_b\) the
noncollision condition gives
\(c\not\equiv b\pmod{k_b}.\)
Hence no modified coordinate eliminates \(c\).  Every integer outside
\(P\) remains eliminated by an unchanged prime coordinate, so the new
survivor set is \(P\setminus B\).
\end{proof}

\medskip
\noindent
The combinatorial part of this lemma is formalized in the supplementary
Lean 4 file
\path{private_composite_coordinates_of_oneClassLargeSieve.lean}.
The one-class large-sieve estimate is assumed there as
\texttt{OneClassLargeSieveBound}; compilation details are given in
\texttt{readme\_Lean\_proof.md}.
\medskip

\begin{lemma}\label{lem:survivor-envelope}
As \(n\to\infty\),
\[
N(n)
\le
|\mathcal W_n|
\le
N(n)
\exp\!\left(
o\!\left(\frac{n}{\log n}\right)
\right).
\]
Equivalently,
\[
\log|\mathcal W_n|
=
\log N(n)
+
o\!\left(\frac{n}{\log n}\right).
\]
\end{lemma}

\begin{proof}
The inclusion \(\mathcal F_n\subseteq\mathcal W_n\) gives
\(N(n)\le|\mathcal W_n|\).  We prove the reverse inequality up to a
subexponential factor.

Fix \(0<\eta<1/2\), and put
\[
J=(\log n)^{1/2-\eta},
\qquad
Y=\left\lfloor\frac{n}{J}\right\rfloor.
\]
Let \(A\in\mathcal W_n\), and choose \(\mathbf a\) such that
\(A\subseteq P:=P_n(\mathbf a)\).  By Lemma~\ref{lem:prime-containers},
there is a profile \(r^{(0)}\) with
\(S_n(r^{(0)})=P\)
whose composite coordinates are redundant lifts of prime coordinates.

Set
\[
B:=
(P\setminus A)\cap\{Y+2,\ldots,n+1\}.
\]
Apply Lemma~\ref{lem:private-composite-coordinates} to \(P\) and \(B\),
using the profile \(r^{(0)}\) constructed above.  Let \(r\) be the
resulting profile, and put
\(S_A:=S_n(r).\)
The conclusion of that lemma gives
\(S_A=P\setminus B.\)
Since \(B\subseteq P\setminus A\), we have
\(A\subseteq S_A.\)
Moreover, by the definition of \(B\),
\[
S_A\cap\{Y+2,\ldots,n+1\}
=
A\cap\{Y+2,\ldots,n+1\}.
\]

Choose one such \(S_A\) for every \(A\in\mathcal W_n\).  For a fixed
\(S\in\mathcal F_n\), every \(A\) mapped to \(S\) agrees with \(S\) above
\(Y+1\) and is a subset of \(S\) on \(X_Y=\{2,\ldots,Y+1\}\).  Hence the
fiber over \(S\) has size at most
\[
2^{|S\cap X_Y|}.
\]
The restriction \(S\cap X_Y\) belongs to \(\mathcal F_Y\), because the
survival of every \(m\le Y+1\) depends only on the coordinates \(r_k\)
with \(k<m\le Y+1\).  Hence the individual-set bound used above gives
\[
|S\cap X_Y|\le K(Y)\ll\frac{Y}{\log Y}.
\]
Therefore
\[
|\mathcal W_n|
\le
N(n)\exp\!\left(O\!\left(\frac{Y}{\log Y}\right)\right)
=
N(n)\exp\!\left(o\!\left(\frac{n}{\log n}\right)\right).
\]
This proves the lemma.
\end{proof}

\begin{proposition}[Comparison through \(\mathcal W_n\)]
\label{prop:two-envelope-comparisons}
As \(n\to\infty\),
\[
A_{\rm adm}(n)
\le |\mathcal W_n|
\le
A_{\rm adm}(n)
\exp\!\left(
o\!\left(\frac{n}{\log n}\right)
\right),
\]
and
\[
N(n)
\le |\mathcal W_n|
\le
N(n)
\exp\!\left(
o\!\left(\frac{n}{\log n}\right)
\right).
\]
\end{proposition}

\begin{proof}
The first pair of inequalities is exactly the content of
Lemma~\ref{lem:admissible-envelope} together with the inclusion
\(A_{\rm adm}(n)\le |\mathcal W_n|\) proved there.  The second pair follows
from Lemma~\ref{lem:survivor-envelope} and the inclusion
\(\mathcal F_n\subseteq\mathcal W_n\) of
Lemma~\ref{lem:prime-containers}.
\end{proof}

\begin{theorem}[Asymptotic comparison with prime-admissible sets]\label{thm:admissible-equivalence}
As \(n\to\infty\),
\[
\log N(n)
=
\log A_{\rm adm}(n)
+
o\!\left(\frac{n}{\log n}\right).
\]
\end{theorem}

\begin{proof}
By Proposition~\ref{prop:two-envelope-comparisons},
\[
\log A_{\rm adm}(n)
\le
\log |\mathcal W_n|
\le
\log A_{\rm adm}(n)
+
o\!\left(\frac{n}{\log n}\right),
\]
while also
\[
\log N(n)
\le
\log |\mathcal W_n|
\le
\log N(n)
+
o\!\left(\frac{n}{\log n}\right).
\]
Hence both \(\log A_{\rm adm}(n)\) and \(\log N(n)\) differ from
\(\log|\mathcal W_n|\) by
\(o(n/\log n)\), and therefore
\[
\log N(n)
=
\log A_{\rm adm}(n)
+
o\!\left(\frac{n}{\log n}\right).
\]
\end{proof}

Let \(X_{\mathbb P}\) denote the full \(\mathbb P\)-admissible subshift,
whose points have support omitting a residue class modulo every prime.
Zero-extending a finite admissible support gives a point of
\(X_{\mathbb P}\), and restricting a point gives the converse.  After
translation from \(X_n\) to an interval of length \(n\), this shows that
\[
A_{\rm adm}(n)=\operatorname{cpx}_{X_{\mathbb P}}(n),
\]
where \(\operatorname{cpx}_{X_{\mathbb P}}(n)\) denotes block complexity.

Kasjan, Lema\'nczyk, and Zuniga Alterman
\cite[Theorem~1.1]{KasjanLemanczykZuniga} proved
\[
(2+o(1))^{n/\log n}
\ll
\operatorname{cpx}_{X_{\mathbb P}}(n)
\ll
(4+o(1))^{n/\log n}.
\]
Thus Theorem~\ref{thm:admissible-equivalence} identifies the constant
window in Corollary~\ref{cor:constant-window} with the corresponding
block-complexity problem for the \(\mathbb P\)-admissible subshift.

\begin{remark}[A barrier from narrow admissible tuples]\label{rem:Hk-barrier}
Let \(H(k)\) denote the minimum diameter of an admissible \(k\)-tuple.
The bounds proved in the cited 2014 work are
\[
\left(\frac12+o(1)\right)k\log k
\le H(k)\le
(1+o(1))k\log k;
\]
see \cite[Theorem~17 and the paragraph immediately following it]{DHJPolymath}. If one could prove
\[
\log A_{\rm adm}(n)
\le
(\log2+o(1))\frac{n}{\log n},
\]
then every admissible \(k\)-tuple of diameter \(H(k)\), together with all
of its \(2^k\) subsets, would imply
\[
H(k)\ge(1-o(1))k\log k.
\]
Hence the conjectural lower endpoint \(\log2\) for \(N(n)\) would improve the lower bound to the leading constant in the cited upper bound for \(H(k)\).
\end{remark}

\section{The structural recurrence and extendibility}
\label{sec:structural-recurrence}
\subsection{The transition from \texorpdfstring{\(n\) to \(n+1\)}{n to n+1}}

The key to the transition from $n$ to $n+1$ is the notion of extendibility.

\begin{definition}\label{def:extendible}
An element $A\in\mathcal F_n$ is called \emph{extendible} if there exists a truncated profile
\[
r=(r_2,\dots,r_n)
\]
such that
\[
S_n(r)=A
\]
and, in addition,
\[
r_k\not\equiv n+2 \pmod{k}
\qquad (2\le k\le n).
\]
Equivalently, $A$ is extendible if and only if $A\cup\{n+2\}\in\mathcal F_{n+1}$.
We write $\mathcal E_n\subseteq\mathcal F_n$ for the family of extendible sets and
\[
E(n):=|\mathcal E_n|.
\]
\end{definition}

\begin{theorem}\label{thm:structural-recurrence}
For every integer $n\ge 1$,
\[
\mathcal F_{n+1}
=
\mathcal F_n\sqcup\{A\cup\{n+2\}: A\in\mathcal E_n\}.
\]
Consequently,
\[
N(n+1)=N(n)+E(n).
\]
\end{theorem}

\begin{proof}
Let
\((r_2,\dots,r_n,r_{n+1})\)
be a profile of length \(n+1\).  For \(m\le n+1\), survival of \(m\)
depends only on the coordinates with index \(k<m\).  Hence
\[
S_{n+1}(r_2,\dots,r_n,r_{n+1})\cap\{2,\dots,n+1\}
=
S_n(r_2,\dots,r_n).
\]
Thus every member of \(\mathcal F_{n+1}\) is either \(A\) or
\(A\cup\{n+2\}\) for some \(A\in\mathcal F_n\).

The point \(n+2\) survives precisely when
\[
r_k\not\equiv n+2\pmod{k}
\qquad (2\le k\le n+1).
\]
The condition for \(k=n+1\) can always be met by choosing
\(r_{n+1}\not\equiv1\pmod{n+1}\).  Hence \(n+2\) can be retained exactly
when the truncated profile realizes an element of \(\mathcal E_n\).

Conversely, if \(A\in\mathcal F_n\), a profile realizing \(A\) can be
extended with \(r_{n+1}\equiv1\pmod{n+1}\), which eliminates \(n+2\).
If \(A\in\mathcal E_n\), choose a realizing profile under which \(n+2\)
survives and take \(r_{n+1}\not\equiv1\pmod{n+1}\).  Therefore
\[
\mathcal F_{n+1}
=
\mathcal F_n\sqcup
\{A\cup\{n+2\}:A\in\mathcal E_n\}.
\]
Taking cardinalities gives
\[
N(n+1)=N(n)+E(n).
\]
\end{proof}

\subsection{An exact covering criterion}

Extendibility admits the following combinatorial characterization.

\begin{definition}
Let $A\subseteq\{2,\dots,n+1\}$. For each $k\in\{2,\dots,n\}$ define the set of \emph{admissible residues}
\[
\Omega_k^+(A):=
(\mathbb Z/k\mathbb Z)\setminus
\Bigl(
\{m\bmod k:\ m\in A,\ m>k\}
\cup
\{(n+2)\bmod k\}
\Bigr).
\]
Thus $\Omega_k^+(A)$ consists of the residue classes modulo $k$ that do not eliminate any element of $A$ larger than $k$, and do not eliminate the new point $n+2$.
\end{definition}

\begin{definition}
For $a\in\Omega_k^+(A)$, define the corresponding \emph{covered block of excluded points}
\[
C_{k,a}(A):=
\{m\in\{2,\dots,n+1\}\setminus A:\ m>k,\ m\equiv a \pmod{k}\}.
\]
\end{definition}

\begin{theorem}\label{thm:extendibility-criterion}
Let \(A\in\mathcal F_n\). Then \(A\) is extendible if and only if there exists a choice
\[
a_k\in\Omega_k^+(A)\qquad (2\le k\le n)
\]
such that
\[
\{2,\dots,n+1\}\setminus A
\subseteq
\bigcup_{k=2}^{n} C_{k,a_k}(A).
\]
\end{theorem}

\begin{proof}
Suppose first that $A$ is extendible. Then there exists a profile
\(r=(r_2,\dots,r_n)\)
with $S_n(r)=A$ and with $n+2$ surviving. By definition of survival of $A$ and $n+2$, we have
\[
r_k\in\Omega_k^+(A)
\qquad (2\le k\le n).
\]
Put $a_k=r_k$. If $m\notin A$, then since $S_n(r)=A$, there exists some $k<m$ such that
\[
m\equiv r_k \pmod{k}.
\]
Hence $m\in C_{k,a_k}(A)$. Therefore every excluded point belongs to the union
\[
\bigcup_{k=2}^{n} C_{k,a_k}(A).
\]

Conversely, suppose there exists a choice $a_k\in\Omega_k^+(A)$ such that
\[
\{2,\dots,n+1\}\setminus A
\subseteq
\bigcup_{k=2}^{n} C_{k,a_k}(A).
\]
Define a profile by $r_k=a_k$. Since $a_k\in\Omega_k^+(A)$, no element of $A$ larger than $k$ is eliminated by the congruence modulo $k$, and neither is $n+2$. On the other hand, every element outside $A$ is eliminated by at least one of the chosen congruences by the covering hypothesis. Hence
\(S_n(r)=A\)
and $n+2$ survives. Therefore $A$ is extendible.
\end{proof}

Two immediate corollaries are useful.

\begin{corollary}[Local obstruction]\label{cor:local-obstruction}
Let \(A\in\mathcal F_n\). If there exists \(k\in\{2,\dots,n\}\) such that
\[
\Omega_k^+(A)=\varnothing,
\]
then $A$ is not extendible.
\end{corollary}

\begin{proof}
If $\Omega_k^+(A)=\varnothing$, there is no admissible choice at modulus $k$, so the criterion in Theorem~\ref{thm:extendibility-criterion} cannot be satisfied.
\end{proof}

\begin{definition}
Let \(A\subseteq\{2,\dots,n+1\}\). For
\(m\in\{2,\dots,n+1\}\setminus A\), define the set of \emph{witness moduli}
\[
W_A(m):=\{k\in\{2,\dots,m-1\}: m\bmod k \in \Omega_k^+(A)\}.
\]
\end{definition}

\begin{corollary}[Necessary witness condition]\label{cor:witness}
Let \(A\in\mathcal F_n\). If \(A\) is extendible, then
\[
W_A(m)\neq\varnothing
\qquad\text{for every } m\in\{2,\dots,n+1\}\setminus A.
\]
\end{corollary}

\begin{proof}
If $A$ is extendible, Theorem~\ref{thm:extendibility-criterion} provides a choice $a_k\in\Omega_k^+(A)$ whose covered blocks contain every excluded point. Hence each excluded point $m$ belongs to some $C_{k,a_k}(A)$, and therefore $m\bmod k\in\Omega_k^+(A)$ for that $k$. Thus $k\in W_A(m)$.
\end{proof}

\begin{remark}
Corollary~\ref{cor:witness} gives a necessary condition but not a sufficient one. The issue is that distinct excluded points may require incompatible residue choices at the same modulus.
\end{remark}

\begin{example}\label{ex:ext-vs-nonext}
Let $n=5$ and work in the ambient set $\{2,3,4,5,6\}$, with new point $7$.

For
\(A=\{2,3,5\},\)
we have
\[
\Omega_2^+(A)=\{0\},\qquad
\Omega_3^+(A)=\{0\},\qquad
\Omega_4^+(A)=\{0,2\},\qquad
\Omega_5^+(A)=\{0,1,3,4\}.
\]
The excluded points are $\{4,6\}$. Choosing $a_2=0$ already gives
\(C_{2,0}(A)=\{4,6\},\)
so the covering condition is satisfied and $A$ is extendible.

For
\(A=\{2,4,6\},\)
the residues of the surviving points larger than $2$ together with the residue of $7$ fill both residue classes modulo $2$, so
\(\Omega_2^+(A)=\varnothing.\)
Hence $A$ is not extendible by Corollary~\ref{cor:local-obstruction}.
\end{example}

\section{Algorithms and computations}
\label{sec:algorithms}
\subsection{Exact dynamic enumeration}

The survivor set is the complement in \(X_n\) of the union of the local
elimination sets.  This observation gives an exact dynamic algorithm that
deduplicates partial unions after every modulus.

\begin{proposition}[Dynamic union recurrence]\label{prop:dynamic-enumeration}
Fix \(n\ge1\), and define
\[
\mathcal U_{1,n}:=\{\varnothing\}.
\]
For \(2\le j\le n\), define recursively
\[
\mathcal U_{j,n}:=
\left\{
U\cup E^{(n)}_{j,a}:
U\in\mathcal U_{j-1,n},\ a\in\mathbb Z/j\mathbb Z
\right\}.
\]
Then
\[
\mathcal F_n=
\{X_n\setminus U:U\in\mathcal U_{n,n}\},
\qquad
N(n)=|\mathcal U_{n,n}|.
\]
The same recurrence may use one representative of each locally distinct set
\(E^{(n)}_{j,a}\) without changing its output.
\end{proposition}

\begin{proof}
After the moduli \(2,\ldots,j\) have been chosen, the set of eliminated points
is exactly a union
\[
E^{(n)}_{2,r_2}\cup\cdots\cup E^{(n)}_{j,r_j}.
\]
The recursive definition of \(\mathcal U_{j,n}\) therefore lists precisely all
such partial unions, with duplicates removed because \(\mathcal U_{j,n}\) is a
set.  At \(j=n\), taking complements in \(X_n\) gives all survivor sets.
Replacing congruent choices that determine the same local elimination set does
not change any union.
\end{proof}

For computation, subsets of \(X_n\) may be represented by bit vectors.
At stage \(j\), form the bitwise unions of each current vector with the
\(c_{j,n}\) locally distinct vectors for modulus \(j\), and remove
duplicates.  Extendibility can then be tested either from
Definition~\ref{def:extendible} or from
Theorem~\ref{thm:extendibility-criterion}.

\subsection{Initial values}

The following values are obtained by exact enumeration and reproduce the initial terms of \seqnum{A396595}. The local counts \(P(n)\) are included to display the collapse from locally distinct profiles to distinct outputs.

\begin{table}[ht]
\centering
\small
\begin{tabular}{c|cccccccccccc}
$n$ & 1 & 2 & 3 & 4 & 5 & 6 & 7 & 8 & 9 & 10 & 11 & 12 \\
\hline
$N(n)$ & 1 & 2 & 3 & 4 & 6 & 8 & 11 & 14 & 17 & 22 & 30 & 38 \\
$P(n)$ & 1 & 2 & 4 & 12 & 36 & 144 & 576 & 2880 & 14400 & 86400 & 518400 & 3628800
\end{tabular}
\caption{Distinct survivor sets $N(n)$ and local equivalence classes $P(n)$.}
\label{tab:NnPn}
\end{table}

\begin{table}[ht]
\centering
\begin{tabular}{c|cccccccccc}
$n$ & 1 & 2 & 3 & 4 & 5 & 6 & 7 & 8 & 9 & 10 \\
\hline
$E(n)$ & 1 & 1 & 1 & 2 & 2 & 3 & 3 & 3 & 5 & 8
\end{tabular}
\caption{The number $E(n)$ of extendible survivor sets in $\mathcal F_n$.}
\label{tab:En}
\end{table}

The gap between the local count \(P(n)\) and the number of distinct
survivor sets is already substantial for small \(n\).  For example,
\[
P(8)=2880,\qquad N(8)=14.
\]
This motivates studying the realizable family \(\mathcal F_n\) and its
extendible subfamily \(\mathcal E_n\), rather than the local count alone.

The values in Table~\ref{tab:En}, together with the recurrence
\[
N(n+1)=N(n)+E(n),
\]
recover the initial segment of Table~\ref{tab:NnPn}. For instance,
\[
N(6)=N(5)+E(5)=6+2=8,
\]
and similarly
\[
N(11)=N(10)+E(10)=22+8=30.
\]

\begin{example}\label{ex:F5}
For \(n=5\), exact enumeration gives
\[
\mathcal F_5=
\bigl\{
\{2,3,5\},
\{2,3\},
\{2,4,6\},
\{2,4\},
\{2,6\},
\{2\}
\bigr\},
\]
so \(N(5)=6\).
\end{example}

\begin{example}\label{ex:E5}
The new point at the transition from \(n=5\) to \(n=6\) is \(7\).  Among the
six elements of \(\mathcal F_5\), the extendible ones are
\[
\{2,3,5\},\qquad \{2,3\}.
\]
Thus
\(E(5)=2,\)
and Theorem~\ref{thm:structural-recurrence} gives
\[
N(6)=N(5)+E(5)=6+2=8.
\]
\end{example}

The recurrence also gives
\(E(n)=N(n+1)-N(n);\)
for example, \(E(10)=8\) and \(N(10)=22\) yield \(N(11)=30\).

\subsection{Comparison with prime-admissible subsets}

For small \(n\), we computed \(A_{\rm adm}(n)\) by enumerating all
subsets of \(X_n\) and testing whether each prime modulus omits at least
one residue class.  The values of \(N(n)\) were computed from the dynamic
recurrence in Proposition~\ref{prop:dynamic-enumeration}.

Define
\[
\Delta(n):=
\frac{\log A_{\rm adm}(n)-\log N(n)}{n/\log n}.
\]

\begin{table}[ht]
\centering
\begin{tabular}{c|r|r|c|c}
\(n\) &
\(N(n)\) &
\(A_{\rm adm}(n)\) &
\(A_{\rm adm}(n)/N(n)\) &
\(\Delta(n)\)\\
\hline
5  & 6   & 10  & 1.667 & 0.164\\
10 & 22  & 45  & 2.045 & 0.165\\
15 & 76  & 170 & 2.237 & 0.145\\
20 & 237 & 535 & 2.257 & 0.122
\end{tabular}
\caption{Exact comparison between survivor sets and prime-admissible
subsets.}
\label{tab:admissible-comparison}
\end{table}

The ordinary ratio \(A_{\rm adm}(n)/N(n)\) need not tend to \(1\).
Theorem~\ref{thm:admissible-equivalence} asserts only that
\[
\log\frac{A_{\rm adm}(n)}{N(n)}
=
o\left(\frac{n}{\log n}\right).
\]
The values of \(\Delta(n)\) in Table~\ref{tab:admissible-comparison}
increase between \(n=5\) and \(n=10\) and decrease for \(n=15,20\).
This range is too short to indicate an asymptotic trend.  The computations
are implemented in the ancillary file \texttt{enumeration.py}.

\section{Open problems}
\label{sec:open-problems}

The results of this paper leave two main asymptotic questions.

\begin{enumerate}

\item
Determine the asymptotic behavior of the extension ratio
\[
\frac{E(n)}{N(n)}.
\]
In particular, is it true that
\[
\frac{E(n)}{N(n)}\longrightarrow 0
\qquad (n\to\infty)?
\]
More generally, determine the correct order of magnitude of
\(E(n)/N(n)\).

\item
Determine the exponential growth constant of \(N(n)\).  Does the
limit
\[
\lim_{n\to\infty}
\frac{\log N(n)}{n/\log n}
\]
exist, and if so, is it equal to \(\log 2\)?

By Theorem~\ref{thm:admissible-equivalence}, the same question may be
formulated in terms of \(A_{\rm adm}(n)\).  The results of this paper give
\[
\log 2
\le
\liminf_{n\to\infty}
\frac{\log N(n)}{n/\log n}
\le
\limsup_{n\to\infty}
\frac{\log N(n)}{n/\log n}
\le
2\log 2.
\]
As noted in Remark~\ref{rem:Hk-barrier}, an upper bound
\[
\log A_{\rm adm}(n)
\le
(\log 2+o(1))\frac{n}{\log n}
\]
would imply \(H(k)\sim k\log k\).

\end{enumerate}

\appendix

\section{Global profiles and finite restrictions}
\label{subsec:global-profile}
The finite model is the main subject of this paper, but a global formulation
clarifies both the origin of the construction and the meaning of the counting
function $N(n)$. Let
\[
\Omega:=\prod_{k\geq 2}\mathbb Z/k\mathbb Z
\]
be the space of global residue profiles.  For
\(F=(F_k)_{k\geq 2}\in\Omega,\)
define the corresponding global survivor set by
\[
\mathcal S(F):=
\left\{
m\geq 2:
m\not\equiv F_k\pmod k
\text{ for every }2\leq k<m
\right\}.
\]

For fixed \(m\), membership in \(\mathcal S(F)\) depends only on
\(F_2,\ldots,F_{m-1}\).  In particular, \(F_k\) affects only candidates
\(m>k\), excluding those satisfying
\(m\equiv F_k\pmod{k}\).
Consequently, a single global profile \(F\) determines an infinite binary
survival pattern, whose support is exactly \(\mathcal S(F)\).

If $r=(r_2,\ldots,r_n)$ is the finite profile obtained by restricting $F$
to the coordinates $2,\ldots,n$, then
\[
S_n(r)=\mathcal S(F)\cap\{2,\ldots,n+1\}.
\]

\begin{proposition}[A profile for upper twin primes]
\label{prop:coded-twin}
Let $F^{\mathrm{tw}}$ be defined by
\[
F^{\mathrm{tw}}_n=
\begin{cases}
0, & \text{if $n$ is prime or $n$ is even},\\
2, & \text{if $n$ is odd and composite}.
\end{cases}
\]
Then
\[
\mathcal S(F^{\mathrm{tw}})
=
\{2,3\}\cup
\{q\in\mathbb P:q-2\in\mathbb P\}.
\]
When listed increasingly, the survivor sequence is obtained by prefixing
$2,3$ to \seqnum{A006512}, the sequence of upper members of twin-prime pairs.
\end{proposition}

\begin{proof}
If $m$ is composite and even, then $m\equiv0\pmod2$, so it is eliminated
at the prime modulus $2$.  If $m$ is composite and odd, let $p$ be its
least prime divisor.  Then $p<m$ and $m\equiv0\pmod p$, so the prime
modulus $p$ eliminates $m$. Hence every survivor is prime.

The integers \(2,3,5\) survive by direct inspection. Now let $q>5$ be prime. Prime moduli do not eliminate $q$.  Even composite
moduli carry residue $0$, so they do not eliminate an odd prime. Thus only
odd composite moduli matter, and they all carry residue $2$. Therefore
$q$ is eliminated if and only if there exists an odd composite $n<q$ such
that
\(q\equiv2\pmod n,\)
or equivalently $n\mid(q-2)$.  Since $q-2$ is odd, such an odd composite
divisor exists if and only if $q-2$ is composite.  Hence $q$ survives
exactly when $q-2$ is prime.
\end{proof}

For this profile, the survivors up to \(50\) are
\[
2,3,5,7,13,19,31,43.
\]
Since the profile itself is defined using primality and compositeness, this characterization has no implication for the infinitude of twin primes.

Define the global survivor family
\[
\mathfrak M:=
\{\mathcal S(F):F\in\Omega\}.
\]

\begin{proposition}[Global interpretation of $N(n)$]
\label{prop:global-interpretation}
For every $n\geq1$,
\[
\mathcal F_n
=
\{\mathcal S(F)\cap X_n:F\in\Omega\}
=
\{S\cap X_n:S\in\mathfrak M\}.
\]
Consequently,
\[
N(n)
=
\left|
\{S\cap X_n:S\in\mathfrak M\}
\right|.
\]
\end{proposition}

\begin{proof}
For $n=1$, the assertion is immediate, since $X_1=\{2\}$ and
$\mathcal F_1=\{\{2\}\}$.  Assume henceforth that $n\geq2$.
Let $S=\mathcal S(F)\in\mathfrak M$, and restrict $F$ to the finite
profile
\[
r=(F_2,\ldots,F_n).
\]
For every $m\in X_n$, all coordinates that can act on $m$ satisfy
\(2\leq k<m\leq n+1,\)
and hence $k\leq n$.  Therefore
\[
S\cap X_n=S_n(r)\in\mathcal F_n.
\]
This proves
\[
\{S\cap X_n:S\in\mathfrak M\}
\subseteq
\mathcal F_n.
\]

Conversely, let $A\in\mathcal F_n$.  By definition, there is a finite
profile
\(r=(r_2,\ldots,r_n)\)
such that $A=S_n(r)$.  Extend $r$ to a global profile $F\in\Omega$ by
setting
\(F_k=r_k \qquad (2\leq k\leq n)\)
and choosing the coordinates $F_k$ arbitrarily for $k>n$.  These later
coordinates cannot affect any $m\in X_n$.  It follows that
\[
\mathcal S(F)\cap X_n=S_n(r)=A.
\]
Thus every member of $\mathcal F_n$ is the restriction of a global
survivor set.  This proves the reverse inclusion and the stated identity
for $N(n)$.
\end{proof}

\begin{corollary}[Global realizability criterion]
\label{cor:global-realizability}
Let $T\subseteq\{2,3,\ldots\}$.  Then
\[
T\in\mathfrak M
\quad\Longleftrightarrow\quad
T\cap X_n\in\mathcal F_n
\qquad\text{for every }n\geq1.
\]
\end{corollary}

\begin{proof}
The forward implication follows from
Proposition~\ref{prop:global-interpretation}.  Conversely, suppose that
$T\cap X_n\in\mathcal F_n$ for every $n\geq1$.  For each $n$, let
$\mathcal T_n$ be the finite nonempty set of truncated profiles
\(r=(r_2,\ldots,r_n)\)
satisfying $S_n(r)=T\cap X_n$; for $n=1$, use the empty profile.
Connect a profile in $\mathcal T_{n+1}$ to its restriction in
$\mathcal T_n$.  This restriction does belong to $\mathcal T_n$,
because the coordinate $n+1$ acts only on candidates larger than
$n+1$.  The resulting rooted tree is infinite and finitely branching.
By K\H{o}nig's infinity lemma, it has an infinite path.  The compatible
residues along this path define a global profile $F\in\Omega$.  For
each $m\geq2$, membership of $m$ is determined by the finitely many
coordinates $2\leq k<m$, so the path gives
\[
m\in\mathcal S(F)
\quad\Longleftrightarrow\quad
m\in T.
\]
Hence $T=\mathcal S(F)\in\mathfrak M$.
\end{proof}

Proposition~\ref{prop:global-interpretation} shows that $N(n)$ counts
the distinct restrictions to $X_n$ of all global survivor sets in
$\mathfrak M$.

A fixed global profile $F$ determines the coherent chain
\[
\mathcal S(F)\cap X_1,\quad
\mathcal S(F)\cap X_2,\quad
\mathcal S(F)\cap X_3,\quad\ldots,
\]
where the restriction at level $n$ is one of the $N(n)$ finite
survivor sets.  Corollary~\ref{cor:global-realizability} shows conversely
that the global survivor sets are exactly the coherent infinite paths
whose restriction at every level belongs to $\mathcal F_n$.

\begin{remark}[Relation with OEIS sequences]
\label{rem:global-oeis-relation}
Let \(T\subseteq\{2,3,\ldots\}\) be the set of values of an increasing sequence recorded in the OEIS. We distinguish three relations with a
profile \(F\):
\[
T=\mathcal S(F),\qquad
|T\mathbin{\triangle}\mathcal S(F)|<\infty,\qquad
T=\mathcal S(F)\cap\mathbb P.
\]
They correspond respectively to exact generation, agreement up to a finite modification, and agreement after restriction to the primes.

The global model is not universal.  For example,
\[
\mathbb P\setminus\{3\}\notin\mathfrak M.
\]
Indeed, eliminating $3$ requires
\(F_2\equiv1\pmod2,\)
but the same coordinate then eliminates every odd prime. Thus Proposition~\ref{prop:global-interpretation} identifies the finite
complexity of the global family without asserting that every integer sequence, or every subsequence of the primes, belongs to that family.
\end{remark}

\section{Additional global survivor profiles}
\label{sec:global-profiles}

This appendix gives further explicit global profiles, some defined by arithmetic predicates and others by periodic residue rules. The propositions establish exact survivor-set identities. These examples are
independent of the asymptotic comparison for the finite model; distributional questions are stated separately.

\begin{definition}
A global profile $F=(F_n)_{n\ge 2}$ is called \emph{eventually periodic} if there
are integers $M\ge 1$ and $n_0\ge 2$, and integers
$c_0,\ldots,c_{M-1}$, such that
\[
F_n\equiv c_{n\bmod M}\pmod n
\qquad (n\ge n_0).
\]
It is called \emph{purely periodic} if such a representation holds with $n_0=2$.
Periodicity here refers to the integer representatives $c_{n\bmod M}$
in these congruences, since the residue classes $F_n$ have varying moduli.
\end{definition}

\subsection{Profiles defined using arithmetic predicates}

\begin{proposition}[A profile for an odd-part condition]\label{prop:coded-oddpart}
Let \(F^{\mathrm{op}}\) be defined by
\[
F^{\mathrm{op}}_n=
\begin{cases}
0, & \text{if }n\text{ is prime or }n\text{ is even},\\
1, & \text{if }n\text{ is odd and composite}.
\end{cases}
\]
For a positive integer $m$, write $m=2^a u$ with $u$ odd, and put
$\operatorname{odd}(m)=u$.  Then
\[
\mathcal S(F^{\mathrm{op}})
=
\{2,3,5\}\cup
\{q\in \mathbb P:\ \operatorname{odd}(q-1)=1
\text{ or } \operatorname{odd}(q-1)\in\mathbb P\}.
\]
\end{proposition}

\begin{proof}
As in the previous proposition, every composite integer is eliminated by a
prime modulus.  Let $q>5$ be prime.  Prime moduli and even composite moduli do
not eliminate $q$, so only odd composite moduli matter, and all of them carry
residue $1$.  Thus $q$ is eliminated if and only if there exists an odd
composite $n<q$ such that
\(q\equiv 1 \pmod n,\)
equivalently, if and only if $n\mid(q-1)$ for some odd composite $n$.  This is
the same as saying that the odd part of $q-1$ is composite.  Therefore $q$
survives precisely when the odd part of $q-1$ is $1$ or an odd prime.  This
gives the asserted formula.
\end{proof}

\begin{proposition}[A profile involving prime powers]\label{prop:coded-prime-powers-one}
Let $F^{(1,\mathrm{pp})}$ be defined by
\[
F^{(1,\mathrm{pp})}_n=
\begin{cases}
0, & \text{if }n\text{ is prime, or }n\text{ is even, or }n\text{ is an odd prime power},\\
1, & \text{if }n\text{ is odd, composite, and not a prime power}.
\end{cases}
\]
Then
\[
\mathcal S(F^{(1,\mathrm{pp})})
=
\{2,3\}\cup\{q\in \mathbb P:\ \operatorname{odd}(q-1)\text{ is }1
\text{ or an odd prime power}\},
\]
where \(\operatorname{odd}(m)\) denotes the odd part of \(m\).
\end{proposition}

\begin{proof}
Every composite integer is again eliminated by a prime modulus.  Let $q>3$ be
prime.  Prime moduli, even composite moduli, and odd prime-power moduli
do not eliminate $q$; an odd composite modulus that is not a prime power
eliminates $q$ exactly when it divides $q-1$.  By
construction, the only odd composite moduli with residue $1$ are those that
are not prime powers.  Hence $q$ survives if and only if every odd composite
divisor of $q-1$ is an odd prime power.  This is equivalent to saying that the
odd part of $q-1$ is either $1$ or an odd prime power.
\end{proof}

\begin{proposition}[A shifted profile involving prime powers]\label{prop:coded-prime-powers-two}
Let $F^{(2,\mathrm{pp})}$ be defined by
\[
F^{(2,\mathrm{pp})}_n=
\begin{cases}
0, & \text{if }n\text{ is prime, or }n\text{ is even, or }n\text{ is an odd prime power},\\
2, & \text{if }n\text{ is odd, composite, and not a prime power}.
\end{cases}
\]
Then
\[
\mathcal S(F^{(2,\mathrm{pp})})
=
\{2,3,5\}\cup
\{q\in\mathbb P\setminus\{2\}:\ \operatorname{odd}(q-2)\text{ is }1
\text{ or an odd prime power}\}.
\]
\end{proposition}

\begin{proof}
Every composite integer is eliminated by a prime modulus.  Let $q>5$ be prime.
Prime moduli, even composite moduli, and odd prime-power moduli do not
eliminate $q$; an odd composite modulus that is not a prime power
eliminates $q$ exactly when it divides $q-2$.  By
construction, those odd composite moduli carrying residue $2$ are precisely
the ones which are not prime powers.  Therefore $q$ survives if and only if
every odd composite divisor of $q-2$ is an odd prime power, which is
equivalent to the stated condition on the odd part of $q-2$.
\end{proof}

\begin{proposition}[A family for prime pairs of fixed gap]\label{prop:gap-2h-family}
Let \(h\ge 2\) be an integer, and define the profile
$F^{(h)}=(F_n^{(h)})_{n\ge 2}$ by
\[
F_n^{(h)}=
\begin{cases}
0, & \text{if $n$ is prime or even},\\
2h, & \text{if $n$ is odd and composite}.
\end{cases}
\]
Then every composite integer $m>2$ is excluded.  Moreover, for every prime
$q>2h+1$,
\[
q\in \mathcal S(F^{(h)})
\iff
q-2h\in\mathbb P.
\]
Hence, apart from a finite initial set of small primes, the survivor sequence
$\mathcal S(F^{(h)})$ is precisely the sequence of upper primes in prime pairs
of gap $2h$.
\end{proposition}

\begin{proof}
The exclusion of composite integers is the same as in
Proposition~\ref{prop:coded-twin}: an even composite is excluded modulo $2$,
and an odd composite is excluded modulo any proper prime divisor.  Let
$q>2h+1$ be prime.  An index carrying residue $0$ cannot exclude $q$,
so an exclusion must come from an odd composite $n<q$ carrying residue
$2h$.  This occurs exactly when $n\mid(q-2h)$.  The positive odd integer
$q-2h$ is less than $q$; it has such a composite divisor if and only if it
is composite.  Therefore $q$ survives exactly when $q-2h$ is prime.
\end{proof}

\begin{remark}
The same construction with \(h=1\) is the twin-prime profile
\(F^{\mathrm{tw}}\) of Proposition~\ref{prop:coded-twin}.
\end{remark}

\begin{proposition}[A profile for safe primes]
\label{prop:safe-primes-family}
Define the profile
\(F^{\mathrm{safe}}=(F_n^{\mathrm{safe}})_{n\ge2}\) by
\[
F_n^{\mathrm{safe}}=
\begin{cases}
1, & \text{if \(n\) is even and \(n/2\) is composite},\\
0, & \text{otherwise}.
\end{cases}
\]
Then
\[
\mathcal S(F^{\mathrm{safe}})
=
\{2,3\}\cup
\{q\in\mathbb P:\ (q-1)/2\in\mathbb P\}.
\]
Moreover, the transformed sequence
\[
\left(
\frac{q-1}{2}
\right)_
{q\in\mathcal S(F^{\mathrm{safe}}),\,q>3}
\]
is precisely the Sophie Germain prime sequence \seqnum{A005384}.
\end{proposition}

\begin{proof}
As in Proposition~\ref{prop:coded-twin}, every composite integer $m>3$ is
excluded: even composites are excluded modulo $2$, and odd composites modulo
a proper prime divisor.  Let $q>3$ be prime.  Every index carrying residue
$0$ cannot exclude $q$.  The only indices carrying residue $1$ are $n=2d$,
where $d$ is composite.
Such an index excludes $q$ exactly when $d\mid(q-1)/2$.  Hence $q$ is
excluded exactly when $(q-1)/2$ is composite, and it survives exactly when
$(q-1)/2$ is prime.  The final statement follows by writing $q=2p+1$.
\end{proof}

\subsection{Elementary identities}

The following elementary example uses a constant residue rule.

\begin{proposition}[A shift of the primes]\label{prop:minus-one-profile}
Let \(F^{(-1)}\) be defined by
\[
F^{(-1)}_n=n-1 \qquad (n\ge2).
\]
Then
\[
\mathcal S(F^{(-1)})
=
\{p-1:\ p\in\mathbb P,\ p\ge3\}.
\]
In particular, this survivor set is not a prime sequence.
\end{proposition}

\begin{proof}
An integer $m\ge 2$ belongs to $\mathcal S(F^{(-1)})$ if and only if
\[
m\not\equiv -1\pmod n
\qquad\text{for every }2\le n<m.
\]
Equivalently, there is no $n$ with $2\le n<m$ dividing $m+1$.  This happens
exactly when $m+1$ is prime.
\end{proof}

\subsection{Pure periodic profiles}

The defining rules for the following profiles are independent of primality,
compositeness, and factorization.

\begin{proposition}[A pure periodic profile for Fermat-type numbers]
\label{prop:fermattype-family}
Define the profile
\(F^{\mathrm{Fer}}=(F_n^{\mathrm{Fer}})_{n\ge2}\) by
\[
F_n^{\mathrm{Fer}}=
\begin{cases}
1, & \text{if \(n\) is odd},\\
0, & \text{if \(n\) is even}.
\end{cases}
\]
Then
\[
\mathcal S(F^{\mathrm{Fer}})
=
\{2\}\cup\{2^k+1:\ k\ge1\}.
\]
\end{proposition}

\begin{proof}
Let $m>2$.

If $m$ is even, then $F_2^{\mathrm{Fer}}=0$, and since $m\equiv 0 \pmod 2$, the index $n=2$
excludes $m$.

Now suppose that $m$ is odd.  Then no even index can exclude $m$, since
$F_n^{\mathrm{Fer}}=0$ for even $n$ and an even modulus cannot divide an odd integer.  The
only possible exclusions come from odd indices $n<m$, for which $F_n^{\mathrm{Fer}}=1$.  Such
an index excludes $m$ exactly when
\(m\equiv 1 \pmod n,\)
that is,
\[
n\mid (m-1).
\]

If $m-1$ has an odd divisor $d>1$, then $d<m$, the index $n=d$ is odd, and it
excludes $m$.

Conversely, if $m-1$ is a power of $2$, then it has no odd divisor $>1$, so no
odd index can exclude $m$, and hence $m$ survives.

Therefore an odd integer $m$ survives if and only if
\(m-1=2^k\)
for some $k\ge 1$, i.e.
\[
m=2^k+1.
\]
Together with the initial survivor $2$, this proves
\[
\mathcal S(F^{\mathrm{Fer}})=\{2\}\cup\{2^k+1:\ k\ge 1\}.
\]
\end{proof}

\begin{remark}
This is a ``Fermat-type'' family in the sense that it produces the full
sequence $2^k+1$, not only those terms that happen to be prime.
\end{remark}

\begin{proposition}[A pure periodic profile for Mersenne-type numbers]
\label{prop:mersennetype-family}
Define the profile
\(F^{\mathrm{Mer}}=(F_n^{\mathrm{Mer}})_{n\ge2}\) by
\[
F_n^{\mathrm{Mer}}=
\begin{cases}
-1, & \text{if \(n\) is odd},\\
0, & \text{if \(n\) is even}.
\end{cases}
\]
Then
\[
\mathcal S(F^{\mathrm{Mer}})
=
\{2\}\cup\{2^k-1:\ k\ge2\}.
\]
\end{proposition}

\begin{proof}
Let $m>2$.

As in the proof of Proposition~\ref{prop:fermattype-family}, an even $m$ is
excluded modulo $2$, while no even modulus excludes an odd $m$.  For odd
$m$, an odd index $n<m$ excludes $m$ exactly when $n\mid(m+1)$.  Every
odd divisor $d>1$ of the even integer $m+1$ satisfies
$d\leq(m+1)/2<m$, so such a divisor supplies an excluding index.  Thus an
odd $m$ survives exactly when $m+1$ is a power of $2$, which gives the
stated set together with the initial survivor $2$.
\end{proof}

\begin{remark}
This is the exact analogue of Proposition~\ref{prop:fermattype-family} for
Mersenne-type numbers.
\end{remark}

\begin{proposition}[A pure periodic profile for powers of two]
\label{prop:powers-of-two-family}
Define the profile
\(F^{\mathrm{pow2}}=(F_n^{\mathrm{pow2}})_{n\ge2}\) by
\[
F_n^{\mathrm{pow2}}=
\begin{cases}
0, & \text{if \(n\) is odd},\\
1, & \text{if } n\equiv 2 \pmod 4,\\
-1, & \text{if } n\equiv 0 \pmod 4.
\end{cases}
\]
Then
\[
\mathcal S(F^{\mathrm{pow2}})=\{2^k:\ k\ge1\}.
\]
\end{proposition}

\begin{proof}
We first show that every power of $2$ survives.  Let $m=2^k$ with $k\ge 1$,
and let $2\le n<m$.

If $n$ is odd, then $F_n^{\mathrm{pow2}}=0$.  Since $\gcd(2^k,n)=1$, we have
\(2^k \not\equiv 0 \pmod n.\)

If $n$ is even, then $F_n^{\mathrm{pow2}}=\pm 1$, hence $F_n^{\mathrm{pow2}}$ is odd.  On the other hand, the
residue of $2^k$ modulo an even modulus is always even.  Therefore
\[
2^k \not\equiv F_n^{\mathrm{pow2}} \pmod n.
\]
So no index \(n<m\) excludes \(2^k\), and thus
\(2^k\in\mathcal S(F^{\mathrm{pow2}})\).

Conversely, let $m\ge 2$ be an integer that is not a power of $2$.

If $m$ is odd, then $m>2$ and
\[
m\equiv 1 \pmod 2.
\]
Since $F_2^{\mathrm{pow2}}=1$, the index $n=2$ excludes $m$.

Now suppose that $m$ is even but not a power of $2$.  Then
\(m=2^a u\)
for some integer $a\ge 1$ and some odd integer $u>1$.  Since $u<m$, $u$ is
odd, and $F_u^{\mathrm{pow2}}=0$, we have
\(m\equiv 0 \pmod u,\)
so the index $n=u$ excludes $m$.

Therefore the only survivors are the powers of $2$.
\end{proof}

Consider the eventually periodic pattern
\[
(0,0,0,0,0,1,1,1,1,0,0,0,0,1,1,1,1,0,0,0,\dots),
\]
indexed from \(n=1\).  Since profiles here begin at \(n=2\), this gives
the following definition.

\begin{proposition}[A pure periodic profile with a prime-valued survivor set]\label{prop:mod8-family}
Define the profile
\(F^{(8)}=(F_n^{(8)})_{n\ge2}\) by
\[
F_n^{(8)}=
\begin{cases}
0, & \text{if } 2\le n\le 5,\\
1, & \text{if } n\ge 6 \text{ and } n\equiv 0,1,6,7 \pmod 8,\\
0, & \text{if } n\ge 6 \text{ and } n\equiv 2,3,4,5 \pmod 8.
\end{cases}
\]
Then
\[
\mathcal S(F^{(8)})
=
\{2,3,5\}\cup
\bigcup_{\alpha\in\{1,2\}}
\{2^\alpha p+1:\ p,\,2^\alpha p+1\in\mathbb P,\ p\equiv5\pmod8\}.
\]
Equivalently, apart from the initial terms \(2,3,5\), the survivor
sequence consists exactly of those primes \(q\) for which
\(q-1=2^\alpha p\) for some \(\alpha\in\{1,2\}\) and some prime
\(p\equiv5\pmod8\).  The first terms are
\[
2,3,5,11,53,59,107,149,347,587,1019,1109,1307,1493,\ldots.
\]

\end{proposition}

\begin{proof}
We first show that every composite integer $m>2$ is excluded.

If $m$ is even, then $F_2^{(8)}=0$, and
\(m\equiv 0 \pmod 2,\)
so $m$ is excluded.

Now let $m$ be odd and composite.

If $m\equiv 1$ or $7 \pmod 8$, then $m-1\ge 8$ and
\[
m-1\equiv 0 \text{ or } 6 \pmod 8,
\]
so $F_{m-1}^{(8)}=1$.  Since
\(m\equiv 1 \pmod{m-1},\)
the index $n=m-1$ excludes $m$.

If $m\equiv 3$ or $5 \pmod 8$, then not all prime factors of $m$ can be
congruent to $1$ or $7$ modulo $8$, since products of such residues are still
$1$ or $7$ modulo $8$.  Hence $m$ has a prime divisor $r<m$ with
\[
r\equiv 3 \text{ or } 5 \pmod 8.
\]
For such an $r$ we have $F_r^{(8)}=0$, and since $r\mid m$ we obtain
\[
m\equiv 0 \pmod r.
\]
Thus $m$ is excluded.

Therefore every survivor other than $2$ must be prime.

Now let $q>5$ be prime.  Since $q$ is prime, no index $n<q$ with $F_n^{(8)}=0$ can
exclude it.  Hence only indices with \(F_n^{(8)}=1\) can exclude \(q\),
and these do so
precisely when
\(q\equiv 1 \pmod n,\)
that is, when
\[
n\mid (q-1).
\]

If $q\equiv 1$ or $7 \pmod 8$, then $q-1\equiv 0$ or $6 \pmod 8$, so
$q-1\ge 6$ and $F_{q-1}^{(8)}=1$.  Since $q\equiv 1 \pmod{q-1}$, the prime $q$ is
excluded.  Hence any surviving prime $q>5$ must satisfy
\[
q\equiv 3 \text{ or } 5 \pmod 8.
\]
Therefore
\(q-1=2^\alpha m\)
with $\alpha\in\{1,2\}$ and $m$ odd.

We claim that $m$ must be a prime congruent to $5 \pmod 8$.

First, if $m$ had an odd prime divisor $r\equiv 1$ or $7 \pmod 8$, then
$F_r^{(8)}=1$ and $r\mid(q-1)$, so $q$ would be excluded.

Next, if $m$ had an odd prime divisor $r\equiv 3 \pmod 4$, then
\(2r\equiv 6 \pmod 8,\)
hence $F_{2r}^{(8)}=1$.  Since $2r\mid(q-1)$, the prime $q$ would again be
excluded.

Thus every odd prime divisor of $m$ must be congruent to $5 \pmod 8$.

Finally, if $m$ had at least two odd prime factors, counted with multiplicity,
then it would have a composite divisor $d\equiv 1 \pmod 8$, because
\[
5\cdot 5\equiv 1 \pmod 8.
\]
For such a divisor $d$ we would have $F_d^{(8)}=1$ and $d\mid(q-1)$, so $q$ would be
excluded.  Therefore $m$ has exactly one odd prime factor and it occurs to the
first power.  Hence
\(m=p\)
for some prime $p\equiv 5 \pmod 8$.

Conversely, suppose that
\(q-1=2^\alpha p\)
with $\alpha\in\{1,2\}$ and $p$ prime congruent to $5 \pmod 8$.
The divisors of $q-1$ are among
\[
1,\ 2,\ 4,\ p,\ 2p,\ 4p.
\]
Now
\[
F_2^{(8)}=0,\qquad F_4^{(8)}=0,\qquad F_p^{(8)}=0,
\]
and since
\[
2p\equiv 2 \pmod 8,\qquad 4p\equiv 4 \pmod 8,
\]
we also have
\[
F_{2p}^{(8)}=0,\qquad F_{4p}^{(8)}=0.
\]
Thus no divisor $n<q$ of $q-1$ satisfies $F_n^{(8)}=1$, so no index can exclude $q$.
Hence \(q\in\mathcal S(F^{(8)})\).

Therefore
\[
\mathcal S(F^{(8)})
=
\{2,3,5\}\cup
\bigcup_{\alpha\in\{1,2\}}
\{2^\alpha p+1:\ p,\,2^\alpha p+1\in\mathbb P,\ p\equiv5\pmod8\}.
\]
\end{proof}

A different period-$8$ rule leads to a multiplicative survivor condition.

\begin{proposition}[A second period-$8$ profile with a prime-valued survivor set]
\label{prop:mod8-bis-family}
Define the profile
$F^{(8_{\mathrm{bis}})}=(F_n^{(8_{\mathrm{bis}})})_{n\geq2}$ by
\[
F_n^{(8_{\mathrm{bis}})}=
\begin{cases}
1, & \text{if } n\equiv4,5,6\pmod8,\\
0, & \text{if } n\equiv0,1,2,3,7\pmod8.
\end{cases}
\]
Let
\[
\mathcal G_8:=
\left\{
u\geq1:
p\mid u\Longrightarrow p\equiv1\pmod8
\text{ for every prime }p
\right\}.
\]
Then
\[
\mathcal S(F^{(8_{\mathrm{bis}})})
=
\{2\}\cup
\{2u+1:\ u\in\mathcal G_8,\ 2u+1\in\mathbb P\}.
\]
Equivalently, apart from the initial survivor $2$, the survivor set
consists exactly of the primes $q$ for which every prime divisor of
$(q-1)/2$ is congruent to $1\pmod8$.  The first terms are
\[
2,3,83,179,227,467,563,1187,1283,1523,1619,1907,\ldots.
\]

\end{proposition}

\begin{proof}
We first show that every composite integer $m>2$ is excluded.  If $m$ is
even, then $F_2^{(8_{\mathrm{bis}})}=0$ and $m\equiv0\pmod2$.

Now let $m$ be odd and composite.  If $m$ has a prime divisor
$r\not\equiv5\pmod8$, then $F_r^{(8_{\mathrm{bis}})}=0$ and
$m\equiv0\pmod r$, so $r$ excludes $m$.  It remains to consider the case
in which every prime divisor of $m$ is congruent to $5\pmod8$.  Let
$\Omega(m)$ denote the number of prime factors of $m$, counted with
multiplicity.  If $\Omega(m)$ is even, then $m\equiv1\pmod8$, and hence
$m\equiv1\pmod4$.  Since $F_4^{(8_{\mathrm{bis}})}=1$, the index $4$
excludes $m$.  If $\Omega(m)$ is odd, then $\Omega(m)\geq3$.  The product
$d$ of any two prime factors of $m$, counted with multiplicity, is a
proper divisor of $m$ satisfying
\[
d\equiv5^2\equiv1\pmod8.
\]
Thus $F_d^{(8_{\mathrm{bis}})}=0$ and $m\equiv0\pmod d$, so $d$ excludes
$m$.  Therefore every survivor other than $2$ is prime.

Let now $q>2$ be prime.  An index $n<q$ with
$F_n^{(8_{\mathrm{bis}})}=0$ cannot exclude $q$, since this would require
$n\mid q$.  Hence $q$ is excluded exactly when $q-1$ has a divisor
congruent to $4$, $5$, or $6$ modulo $8$.

Suppose first that $q$ survives.  Since $F_4^{(8_{\mathrm{bis}})}=1$, we
must have $4\nmid(q-1)$.  Thus
\(q-1=2u\)
with $u$ odd.  Let $p$ be a prime divisor of $u$.  If
$p\equiv5\pmod8$, then $p\mid(q-1)$ and
$F_p^{(8_{\mathrm{bis}})}=1$, a contradiction.  If
$p\equiv3$ or $7\pmod8$, then
\(2p\equiv6\pmod8,\)
and the divisor $2p$ of $q-1$ also excludes $q$.  Consequently every
prime divisor of $u$ is congruent to $1\pmod8$, so $u\in\mathcal G_8$.

Conversely, suppose that $u\in\mathcal G_8$ and
\[
q=2u+1\in\mathbb P.
\]
Every divisor of $u$ is congruent to $1\pmod8$, and every divisor of
$2u$ is congruent to $1$ or $2\pmod8$.  Thus $q-1=2u$ has no divisor
congruent to $4$, $5$, or $6$ modulo $8$.  No index with residue value
$1$ excludes $q$, and, as noted above, no index with residue value $0$
can exclude a prime.  Hence $q\in\mathcal S(F^{(8_{\mathrm{bis}})})$.
\end{proof}

A period-$4$ family gives a third, parametrized mechanism.

\begin{proposition}[A period-$4$ family with a prime-pair tail]
\label{prop:period4-family}
Let $a\geq2$ be even, and define
$F^{(4,a)}=(F_k^{(4,a)})_{k\geq2}$ by
\[
F_k^{(4,a)}=
\begin{cases}
2a, & \text{if } k\equiv1\pmod4,\\
a,  & \text{if } k\equiv2\pmod4,\\
0,  & \text{if } k\equiv3\pmod4,\\
-a, & \text{if } k\equiv0\pmod4,
\end{cases}
\]
where each displayed value is interpreted modulo $k$.  For every odd
integer $m>2$,
\[
m\in\mathcal S(F^{(4,a)})
\]
if and only if
\[
d\nmid m
\qquad
\text{for every }2\leq d<m\text{ with }d\equiv3\pmod4,
\]
and
\[
d\nmid(m-2a)
\qquad
\text{for every }2\leq d<m\text{ with }d\equiv1\pmod4.
\]
Consequently, every composite survivor $m$ satisfies
\[
25\leq m\leq2a+1,
\]
and
\[
\mathcal S(F^{(4,a)})\cap\{2a+2,2a+3,\ldots\}
=
\{p+2a:\ p,p+2a\in\mathbb P,\ p\equiv3\pmod4\}.
\]
In particular, for
\[
a\in\{2,4,6,8,10\},
\]
the survivor set is prime-valued:
\[
\mathcal S(F^{(4,a)})\subseteq\{2\}\cup\mathbb P.
\]
For the five parameters occurring in the motivating computations, the
corresponding survivor sequences begin as follows:
\[
\begin{aligned}
a=2:\quad&
2,3,5,7,11,23,47,71,83,107,131,167,227,311,\ldots;\\
a=4:\quad&
2,3,5,7,11,19,31,67,79,139,199,271,367,439,\ldots;\\
a=6:\quad&
2,3,5,11,13,19,23,31,43,59,71,79,83,139,151,163,\ldots;\\
a=8:\quad&
2,3,5,7,13,17,19,23,47,59,83,167,179,227,239,347,\ldots;\\
a=10:\quad&
2,3,5,7,13,17,19,23,31,43,67,79,103,127,151,199,\ldots.
\end{aligned}
\]

\end{proposition}

\begin{proof}
Since $a$ is even, $F_2^{(4,a)}\equiv0\pmod2$, so every even integer
$m>2$ is excluded.

Let $m>2$ be odd.  If $k$ is even, then $F_k^{(4,a)}$ is even, and hence
$m-F_k^{(4,a)}$ is odd; therefore an even index $k$ cannot exclude $m$.
For odd $k$, the definition gives
\[
F_k^{(4,a)}=
\begin{cases}
2a, & \text{if } k\equiv1\pmod4,\\
0,  & \text{if } k\equiv3\pmod4.
\end{cases}
\]
Thus an index $k\equiv3\pmod4$ excludes $m$ exactly when $k\mid m$,
whereas an index $k\equiv1\pmod4$ excludes $m$ exactly when
$k\mid(m-2a)$.  This proves the stated criterion.

Suppose that an odd composite integer $m$ survives.  It cannot have a
prime divisor congruent to $3\pmod4$, since such a divisor would exclude
$m$.  Hence every prime divisor of $m$ is congruent to $1\pmod4$.
It follows that
\(m\equiv1\pmod4\)
and $m\geq5^2=25$.  If $m>2a+1$, then
\(d:=m-2a\)
satisfies $2\leq d<m$ and $d\equiv1\pmod4$.  Since
$d\mid(m-2a)$, the criterion shows that $d$ excludes $m$, a
contradiction.  Therefore every composite survivor lies in
$[25,2a+1]$.

It remains to identify the tail.  Let $q>2a+1$.  By the preceding
paragraph, any surviving $q$ is prime.  Put
\(p:=q-2a.\)
Then $p>1$ is odd.  If $p$ has a prime divisor congruent to $1\pmod4$,
that divisor excludes $q$.  If every prime divisor of $p$ is congruent
to $3\pmod4$ and $p$ is composite, then the product of two prime
factors of $p$, counted with multiplicity, is a divisor of $p$
congruent to $1\pmod4$ and smaller than $q$; it also excludes $q$.
Hence a surviving $q$ must satisfy
\[
p\in\mathbb P,
\qquad
p\equiv3\pmod4.
\]

Conversely, suppose that $p$ and $q=p+2a$ are prime and
$p\equiv3\pmod4$.  No index congruent to $3\pmod4$ can divide $q$,
and the only positive divisors of $q-2a=p$ are $1$ and $p$, neither of
which is an admissible excluding index congruent to $1\pmod4$.
The criterion therefore gives
$q\in\mathcal S(F^{(4,a)})$, proving the tail identity.

Finally, if $a\in\{2,4,6,8,10\}$, then $2a+1\leq21$, so the interval
$[25,2a+1]$ is empty.  Hence no composite survivor exists.
\end{proof}

\begin{remark}
Propositions~\ref{prop:mod8-family},
\ref{prop:mod8-bis-family}, and
\ref{prop:period4-family} exhibit three distinct mechanisms arising
from short periodic residue rules.  The first period-$8$ family is
governed by prime values of two affine linear forms, the second by the
multiplicative semigroup generated by primes congruent to $1\pmod8$,
and the period-$4$ family has, after a finite initial segment, a
prime-pair tail of gap $2a$.

The prime-valued conclusion in Proposition~\ref{prop:period4-family}
does not hold for every even parameter.  For example,
\[
25\in\mathcal S(F^{(4,12)}).
\]
More generally, if $M$ is composite, every prime divisor of $M$ is
congruent to $1\pmod4$, and
\[
a=\frac{M-1}{2},
\]
then $a$ is even and $M\in\mathcal S(F^{(4,a)})$.  We make no claim
here that the resulting survivor sequences are new in the literature.
\end{remark}

\noindent The defining rules and OEIS relations of these profiles are summarized
in Table~\ref{tab:globalexamples}.

\begin{table}[ht]
\centering
\scriptsize
\setlength{\tabcolsep}{3pt}
\begin{tabular}{p{0.85in}p{1.05in}p{2.20in}p{1.55in}}
Profile & Defining rule & Survivor description & OEIS relation type \\
\hline

$F^{\mathrm{op}}$
& arithmetic predicate
& primes $q$ for which $\operatorname{odd}(q-1)$ is $1$ or prime
& finite modification: $2$ followed by \seqnum{A074781}
\\[4pt]

$F^{(1,\mathrm{pp})}$
& arithmetic predicate
& primes for which the odd part of $q-1$ is $1$ or a prime power
& exact generation: \seqnum{A077500}
\\[4pt]

$F^{(2,\mathrm{pp})}$
& arithmetic predicate
& $2$, together with primes $q>2$ for which the odd part of $q-2$ is $1$ or a prime power
& finite modification: $2,3$ followed by \seqnum{A267945}
\\[4pt]

$F^{(h)}$, $h\geq2$
& arithmetic predicate
& upper primes in pairs of gap $2h$, apart from initial survivors
& finite modification: \seqnum{A046132} for $h=2$;
  \seqnum{A046117} for $h=3$
\\[4pt]

$F^{\mathrm{safe}}$
& arithmetic predicate
& safe primes, apart from the initial survivors $2,3$
& finite modification: $2,3$ followed by \seqnum{A005385}
\\[4pt]

$F^{(-1)}$
& elementary identity
& $p-1$, for primes $p\geq3$
& finite modification: \seqnum{A006093} without its initial term $1$
\\[4pt]

$F^{\mathrm{Fer}}$
& pure periodic
& $2^k+1$, $k\geq0$
& exact generation: \seqnum{A000051}
\\[4pt]

$F^{\mathrm{Mer}}$
& pure periodic
& $2$, followed by $2^k-1$, $k\geq2$
& finite modification: \seqnum{A000225} without $0,1$, with $2$ prefixed
\\[4pt]

$F^{\mathrm{pow2}}$
& pure periodic
& $2^k$, $k\geq1$
& finite modification: \seqnum{A000079} without its initial term $1$
\\[4pt]

$F^{(8)}$
& pure periodic
& primes $q$ with $q-1=2^\alpha p$,
  $\alpha\in\{1,2\}$, $p\equiv5\pmod8$
& \multicolumn{1}{c}{---}
\\[4pt]

$F^{(8_{\mathrm{bis}})}$
& pure periodic
& primes $q$ for which every prime divisor of $(q-1)/2$ is
  congruent to $1\pmod8$
& \multicolumn{1}{c}{---}
\\[4pt]

$F^{(4,a)}$, $a\geq2$ even
& pure periodic
& finite exceptional set followed by primes $q$ such that
  $q-2a\in\mathbb P$ and $q-2a\equiv3\pmod4$
& \multicolumn{1}{c}{---}
\end{tabular}

\caption{Defining rules and verified OEIS relations of the additional global
profiles. A dash means that no OEIS relation is asserted here. The
distributional question associated with $F^{(8)}$ is given in
Conjecture~\ref{conj:mod8-density}.}
\label{tab:globalexamples}
\end{table}

\subsection{Distributional conjectures}

For the profiles in Propositions~\ref{prop:coded-twin},
\ref{prop:gap-2h-family}, and~\ref{prop:safe-primes-family}, infinitude
reduces respectively to infinitude of twin primes, prime pairs of gap $2h$,
and Sophie Germain primes.  For the
period-$4$ family of Proposition~\ref{prop:period4-family}, infinitude of the
tail is equivalent to the existence of infinitely many prime pairs
$p,p+2a$ with $p\equiv3\pmod4$.

For the pure periodic profile \(F^{(8)}\) of
Proposition~\ref{prop:mod8-family}, define
\[
A_8(x):=\#\{q\le x:\ q\in\mathcal S(F^{(8)})\}.
\]
Equivalently,
\[
A_8(x)
=
\#\left\{
q\le x:\ q\in\mathbb P,\ q-1=2^\alpha p,\ \alpha\in\{1,2\},\
p\in\mathbb P,\ p\equiv 5\pmod 8
\right\}
+O(1).
\]

\begin{conjecture}\label{conj:mod8-density}
There exists a constant $C_8>0$ such that
\[
A_8(x)\sim C_8\frac{x}{(\log x)^2}
\qquad (x\to\infty).
\]
\end{conjecture}

\section{Computational reproducibility}
\label{app:enumeration-code}

The ancillary Python file \path{enumeration.py} implements the dynamic
union recurrence of Proposition~\ref{prop:dynamic-enumeration}, exhaustive
enumeration of prime-admissible subsets for the values used in
Table~\ref{tab:admissible-comparison}, and the global survivor membership
test used in Appendix~\ref{subsec:global-profile}.  It uses only the Python
standard library and performs exact integer and bit-vector operations.  For
(n=5,10,15,20), it returns
\[
\begin{array}{c|rrrr}
n&5&10&15&20\\
\hline
N(n)&6&22&76&237\\
A_{\rm adm}(n)&10&45&170&535.
\end{array}
\]
The recorded computation also checks the upper-twin-prime example through
$50$.

\section*{Acknowledgments}
The first author is grateful to the second author for guidance and support
throughout the first author's doctoral studies, and to the Department of
Computer Science at Sapienza University of Rome for its research environment.

\end{document}